\documentclass[11pt, a4paper]{article}
 \usepackage{amsmath}
 \usepackage{amssymb}
\usepackage[margin=21mm]{geometry}
\usepackage[utf8]{inputenc}
\usepackage{graphics}
\usepackage{xspace}
\usepackage{color}
 \usepackage{amsthm}
 \usepackage[old]{old-arrows}
  \usepackage{amsfonts}
\usepackage[dvipsnames]{xcolor}
\usepackage{mathtools}
\usepackage{amssymb}
\usepackage{latexsym}
\newcommand{\R}{{\mathbb R}}

\usepackage[mathscr]{euscript}
\usepackage{mathrsfs}
 \usepackage{enumerate}
 \usepackage{cite}
 \usepackage{tikz}
 \usepackage[pagebackref]{hyperref}
 \usepackage{hyperref}
\newcommand{\eps}{\varepsilon}
\catcode`@=12

\makeatother
\usepackage[english]{babel}

 \usepackage[hyperpageref]{backref}
\newtheorem{theorem}{Theorem}
\newtheorem{remark}[theorem]{Remark}
\newtheorem{lemma}[theorem]{Lemma}
\newtheorem{proposition}[theorem]{Proposition}

\DeclareMathOperator*{\supp}{\text{supp}}

\numberwithin{theorem}{section}
\numberwithin{equation}{section}
\renewcommand{\theequation}{\arabic{section}.\arabic{equation}}
\renewcommand{\thetheorem}{\arabic{section}.\arabic{theorem}}

\title{On the classification of solutions to the Logarithmic Laplacian critical Choquard equation}
\date{}
\author{Rakesh Arora$^{1}$\footnote{ R. Arora, \textit{E-mail address:}
  \texttt{rakesh.mat@iitbhu.ac.in}\thanks{Corresponding author}} \ and Jacques Giacomoni$^{2} \footnote{J.~Giacomoni, \textit{E-mail address:} \texttt{jacques.giacomoni@univ-pau.fr}}$, \\
       \small $^{1}$ Department of Mathematical Sciences, Indian Institute of Technology (IIT-BHU) Varanasi, \\ \small Uttar Pradesh 221005, India\\
\small $^{2}$ LMAP, UMR E2S-UPPA CNRS 5142, Ba\^timent IPRA, Avenue de l’Universit\'e F-64013 Pau, France \\}

\providecommand{\keywords}[1]
{
  \small	
  \textbf{\textit{Keywords---}} #1
}
\begin{document}
\maketitle \vspace{-1.8\baselineskip}
\begin{abstract}
In this work, we establish a sharp logarithmic Choquard inequality by combining Beckner’s entropy inequality with the sharp Pitt's type inequality. Motivated by this estimate, we classify the positive classical solutions, in a suitable integrability class, of the following critical
logarithmic Choquard problem
\begin{equation*}\label{main-Choquard-problem}
    \begin{split}
 \mathcal{L}_\Delta u & = \sigma u + \frac{1}{\|u\|_2^2} \left(G_{\ln}(u) - \frac{4}{N}\left[ \int_{\mathbb{R}^N} u^2 \ln u ~dx \right]\right) u \qquad \text{in} \ \mathbb{R}^N
\end{split}
\end{equation*}
where $\sigma\in \mathbb{R}$, $N \geq 1$, $\mathcal{L}_\Delta$ in the Logarithmic Laplacian and $G_{\ln}(u) =  \ln \left(\frac{1}{|x|^4}\right) \ast u^2.$ We construct
explicit solutions as limits of critical fractional Choquard bubbles and prove that every positive classical solution in the prescribed class has the form
\[
u_{\sigma,t}(x) = e^{\frac{N}{4}\left(\sigma - B_{N, \mathcal{L}}\right)} B_{N,0} \left(\frac{t}{t^2 + |x-x_0|^2}\right)^\frac{N}{2} \ \text{for any} \ t \in \mathbb{R}^+, \ x_0 \in \mathbb{R}^N,
\]
where $B_{N,0}= \left(\frac{\Gamma(N)}{\pi^{\frac{N}{2}}\Gamma\left(\frac{N}{2}\right)}\right)^\frac{1}{2}$ and 
$
B_{N,\mathcal{L}} = \left[2\ln 2 + 4 \psi\left(\frac{N}{2}\right) - 2\psi(N) + \frac{4}{N} \ln (B_{N,0})\right].$
\end{abstract}

\keywords {Logarithmic Laplacian, Choquard Equation, Radial Symmetry, Monotonicity, Uniqueness.}\\

\textbf{Mathematics Subject Classification:} 35J61, 35B06.
\section{Introduction}
Sharp functional inequalities provide a natural starting point for the study of critical elliptic equations: they identify the balance between diffusion and nonlinear interaction and single out distinguished families of solutions. At logarithmic order, power-law expressions are replaced by entropy terms, and the usual homogeneity gives way to additive logarithmic scaling laws (see \cite{Beckner-1995, Beckner-1993, Beckner-1999, Carlen-Loss-1992, Frank-Lieb-2010, Frank-Konig-Tang-2020}). In this work, we establish a sharp inequality coupling logarithmic interaction, entropy, and the quadratic form of the logarithmic Laplacian. Motivated by this inequality, we study the existence, symmetry, and complete classification of positive classical solutions to the critical logarithmic Choquard equation driven by the logarithmic Laplacian,
\begin{equation}\label{main-problem}
\mathcal{L}_\Delta u = \sigma u + \frac{1}{\|u\|_2^2} \left(G_{\ln}(u) - \frac{4}{N}\left[ \int_{\mathbb{R}^N} u^2 \ln u ~dx \right] \right) u \quad \text{in } \mathbb{R}^N
\end{equation}
where $\sigma \in \mathbb{R}$ and $\mathcal{L}_\Delta$ denotes the logarithmic Laplacian in $\mathbb{R}^N$, $N \in \mathbb{N}$, introduced by Chen and Weth \cite{Chen-Weth_2019} and defined as
\begin{equation*}
\mathcal{L}_\Delta u(x) = c_N \int_{\mathbb{R}^N} \frac{u(x)\mathbf{1}_{B_1(x)}(y) - u(y)}{|x-y|^N}\, dy + \rho_N u(x)
\end{equation*}
where
\begin{equation*}
c_N := \pi^{-N/2}\Gamma\!\left(\frac{N}{2}\right) = \frac{2}{\omega_{N-1}}, 
\qquad 
\rho_N := 2\ln 2 + \psi\!\left(\frac{N}{2}\right) - \gamma.
\end{equation*}
Here $\omega_{N-1} := |\mathbb{S}^{N-1}|$ denotes the surface measure of the unit sphere in $\mathbb{R}^N$, $\gamma = -\Gamma'(1)$ is the Euler--Mascheroni constant, and $\psi = \Gamma'/\Gamma$ is the digamma function associated with the Gamma function $\Gamma$.

Over the last few decades, semilinear equations driven by integro-differential operators have attracted considerable attention due to the increasing recognition of nonlocal phenomena and their broad spectrum of applications across mathematics, physics, biology, and other scientific disciplines. This growing interest has stimulated substantial advances in the analysis of nonlocal equations from both analytical and probabilistic perspectives; see, for instance, \cite{Antil-Bartels-2017, Jarohs-Saldana-Weth-2024, Pellacci-Verzini-2018, Sprekels-Valdinoci-2017}. Among the various classes of nonlocal operators, fractional powers of the Laplacian occupy a central position owing to rich mathematical structure, remarkable qualitative properties, and their fundamental role in modeling long-range interactions and anomalous diffusion.

For $s\in(0,1)$, the fractional Laplacian $(-\Delta)^s$ admits a representation as a singular integral operator defined in the principal value sense,
\[
(-\Delta)^s u(x)=c_{N,s}\lim_{\varepsilon\to0^+}
\int_{\mathbb{R}^N\setminus B_{\varepsilon}(x)}
\frac{u(x)-u(y)}{|x-y|^{N+2s}}\,dy, \quad c_{N,s}=2^{2s}\pi^{-N/2}s
\frac{\Gamma\left(\frac{N+2s}{2}\right)}{\Gamma(1-s)}>0
\]
where $c_{N,s}$ is a normalization constant chosen so that, for any function $u\in C_c^{\infty}(\mathbb{R}^N)$,
\[
\mathcal{F}((-\Delta)^s u)(\xi)=|\xi|^{2s}\widehat{u}(\xi)
\quad \text{for all }\xi\in\mathbb{R}^N.
\]
It is well known that the fractional Laplacian approaches classical operators as the order varies. In particular,
\[
\lim_{s\to1^-}(-\Delta)^s u(x)=-\Delta u(x),
\qquad
\lim_{s\to0^+}(-\Delta)^s u(x)=u(x),
\]
for $u\in C_c^2(\mathbb{R}^N)$; see, for example, \cite{DiNezza-Palatucci-Valdinoci_2022}. A further remarkable observation, established in \cite{Chen-Weth_2019}, concerns the expansion of the fractional Laplacian at $s=0$. More precisely, they showed that for $u\in C^\beta_c(\mathbb{R}^N)$ with some $\beta>0$
\begin{equation}\label{main:operator}
    \mathcal{L}_\Delta:=\frac{d}{ds}\Big|_{s=0}(-\Delta)^s \quad \text{and} \quad  \text{$\mathcal{F}(\mathcal{L}_\Delta u)(\xi)=(2\ln|\xi|)\widehat{u}(\xi)\quad \text{for a.e. }\xi\in\mathbb{R}^N$}
\end{equation}
where $\mathcal{F}$ denotes the Fourier transform of $u$ given by $$\mathcal{F}(u)(\xi):= \frac{1}{(2\pi)^{\frac{N}{2}}} \int_{\mathbb{R}^N} e^{- \iota x \cdot \xi} u(x) ~dx, \quad \xi \in \mathbb{R}^N.$$  Since its introduction, this operator has attracted considerable attention, and its associated linear and semilinear theory. We refer ther reader to \cite{Chen-Weth_2019, Arora-Mukherjee-2026, Arora-Mukherjee-Vaishnavi-2026, Chen-Veron-2023, Feulefack-Jarohs-Weth-2022, Jarohs-Weth-2019, Laptev-Weth-2021} (for eigenvalue problems and maximum principle), \cite{Santamaria-Saldana-2022, Arora-Giacomoni-Vaishnavi-2025} (for sharp Logarithmic-Sobolev inequality and  embeddings), \cite{Angeles-Saldana-2023, Santamaria-Saldana-2022} (for the asymptotics of fractional Laplacian problems), \cite{Arora-Giacomoni-Hajaiej-Vaishnavi-2026, Arora-Hajaiej-Perera-2025, Dyda-Jarohs-Sk-2025} (for Logarithmic $p$-Laplacian problems), \cite{Chang_Lara-Saldana-2022, Feulefack-Jarohs-2023,  Santamaria-Rios-Saldana-2024} (for regularity results), \cite{Chen-Hauer-Weth-2023} (for Caffarelli-Silvestre extension type problem). 

A parallel object on the round sphere $(\mathbb{S}^n,g)$ is the conformal fractional Laplacian $\mathscr{P}_g^s$, whose derivative at $s=0$ yields the conformal logarithmic Laplacian $\mathscr{P}_g^{\ln} := \partial_s \mathscr{P}_g^s|_{s=0}$. On $\mathbb{S}^n$, Frank, K\"onig, and Tang \cite{Frank-Konig-Tang-2020} classified the nonnegative solutions of the equation
\begin{equation*}
\int_{\mathbb{S}^n} \frac{u(\xi)-u(\eta)}{|\xi-\eta|^n}\,d\eta = E_n\, u(\xi)\ln u(\xi), 
\qquad \xi \in \mathbb{S}^n, 
\qquad E_n := \frac{4}{n}\frac{\pi^{n/2}}{\Gamma(n/2)},
\end{equation*}
which arises as the Euler-Lagrange equation of a conformal logarithmic Sobolev inequality, that is  
\begin{equation*}
\int\int_{\mathbb{S}^n\times\mathbb{S}^n}\frac{|v(\omega)-v(\eta)|^2}{|\omega-\eta|^n}\mathrm{d}\omega\mathrm{d}\eta\geq C_n\int_{\mathbb{S}^n}|v(\omega)|^2\ln\frac{|v(\omega)|^2|\mathbb{S}^n|}{||v||^2_2}\mathrm{d}\omega
\end{equation*}
with $C_n=\frac{4}{n}\frac{\pi^{n/2}}{\Gamma(n/2)}$, showing that all such solutions take the explicit form
\begin{equation*}
u_\theta(\xi) = \left(\frac{\sqrt{1-|\theta|^2}}{1-\theta\cdot\xi}\right)^{n/2}, 
\qquad \theta \in \mathbb{R}^{n+1},\ |\theta|<1.
\end{equation*}
In $\mathbb{R}^n$, Chen and Zhou \cite{Chen-Zhou_2025}, considered the critical growth equation
\begin{equation*}
\mathcal{L}_\Delta u = k\, u\ln u, \qquad u \ge 0 \ \text{in } \mathbb{R}^n,
\end{equation*}
and showed, via a moving-plane method adapted to the Logarithmic Laplacian $\mathcal{L}_\Delta$ and the non-monotone nonlinearity $t \mapsto t\ln t$, that $k = 4/n$ is the unique critical exponent for which positive solutions exist, all of the bubble form
\begin{equation*}
u_{\tilde x,t}(x) = \beta_n\left(\frac{t}{t^2+|x-\tilde x|^2}\right)^{n/2}, 
\qquad t>0,\ \tilde x \in \mathbb{R}^n, 
\qquad \beta_n = 2^{n/2} e^{\frac{n}{2} \psi(n/2)}.
\end{equation*}

In \cite{Fernandez-Saldana-2025}, Fern\'andez and Salda\~na  unified these two classification results by constructing the conformal logarithmic Laplacian on $\mathbb{S}^n$ intrinsically, as
\begin{equation*}
\mathscr{P}_g^{\ln}u(z) = c_n\int_{\mathbb{S}^n}\frac{u(z)-u(\zeta)}{|z-\zeta|^n}\,dV_g(\zeta) + A_n u(z), 
\qquad A_n = 2\psi(n/2),
\end{equation*}
establishing its spectral decomposition on spherical harmonics and its conformal covariance law. Via the stereographic pullback $\iota$, they showed that the equation $\mathscr{P}_g^{\ln}u = \frac{4}{n}u\ln|u| + \mu u$ on $\mathbb{S}^n$, $\mu \in \mathbb{R}$ is equivalent to
\begin{equation*}
\mathcal{L}_\Delta v = \frac{4}{n} v\ln|v| + \mu v \quad \text{in } \mathbb{R}^n,
\end{equation*}
extending the classification in \cite{Frank-Konig-Tang-2020} to weak solutions, recovering the bubbles in \cite{Chen-Zhou_2025} as a special case. Most recently, Chen, Chen, and Hauer \cite{Chen-Chen-Hauer-2026} interpolated between the fractional and logarithmic regimes by differentiating $\mathscr{P}_g^s$ at a general order $s \in (0,1)$ rather than at $s=0$, defining the conformal fractional--logarithmic Laplacian
\begin{equation}
\mathscr{P}_g^{s+\ln}u := \left.\frac{d}{dt}\right|_{t=s}\mathscr{P}_g^t u,
\end{equation}
and studied the associated Yamabe-type equation on $\mathbb{S}^n$, whose $s\to 0^+$ limit recovers the logarithmic Yamabe equation of \cite{Fernandez-Saldana-2025} and, through it, the classification in \cite{Frank-Konig-Tang-2020} and \cite{Chen-Zhou_2025}. They further used this framework to establish new sharp fractional-logarithmic Sobolev-type inequalities and to show that the naive fractional-logarithmic analogue of the sharp Sobolev inequality fails in general.

Parallel to this, the Choquard equation, in which the nonlinearity is mediated by a Riesz-type convolution kernel rather than a pointwise power, has been extensively studied for the classical and fractional Laplacians, motivated by its appearance in models of self-gravitating quantum particles, nonlinear optics, and the Hartree approximation of many-body quantum systems \cite{Lieb-1977, Moroz-VanSchaftingen-2013, Moroz-VanSchaftingen-2017, Moroz-VanSchaftingen-2015-1, Moroz-VanSchaftingen-2015-2, Moroz-VanSchaftingen-2015-3}. We refer the reader to \cite{Bucur-Cassani-Tarsi-2022, Cassani-Du-Liu-2024, Cingolani-Weth-2022, Cingolani-Weth-2016, Choquard-Stubbe-Vuffray-2008, Masaki-2011} for results on the local elliptic equations with Choquard type nonlinearity and logarithmic kernel. However, the present problem differs from these models because both the diffusion operator and the interaction potential are logarithmic.

The study of critical Choquard equations is closely connected with
sharp inequalities that control the nonlocal interaction energy
in terms of the quadratic form of the diffusion operator.
For the fractional Laplacian, this control follows by combining
the Hardy-Littlewood-Sobolev inequality (see \cite[Theorem 3.1]{Lieb-1983}) with the critical
Sobolev inequality (see \cite[Theorem 6.5]{DiNezza-Palatucci-Valdinoci_2022}). Our aim is to establish the corresponding
sharp estimate at logarithmic order. In this setting, entropy
provides the link between the interaction and diffusion terms:
Beckner's entropy inequality (see Lemma \ref{Beckner-ineq}) bounds the logarithmic interaction
energy in terms of $\int_{\mathbb{R}^N}u^2\ln|u|\,dx$, while the
logarithmic Sobolev inequality (see \cite[Theorem 3]{Beckner-1995}) associated with Beckner's
Pitt-type inequality controls this entropy by the quadratic
form of $\mathcal{L}_\Delta$. Combining these two estimates,
with the appropriate entropy correction, yields a sharp
inequality relating
\[
\frac{1}{\|u\|_2^2}
\int_{\mathbb{R}^N} G_{\ln}(u)u^2\,dx
-\frac{4}{N}\int_{\mathbb{R}^N}u^2\ln \frac{|u(x)|}{\|u\|_2}\,dx
\]
to the logarithmic diffusion energy. The common bubble
extremals of the two constituent inequalities ensure the
sharpness of the resulting bound. Next, we state our first main result, establishing the sharp logarithmic Choquard inequality and characterizes its extremals.
\begin{proposition}\label{new:Log-Choq:ineq}
    Let $u \in L^2(\mathbb{R}^N)$ such that
\begin{equation}\label{reg:weak-solu}
    \int_{\mathbb{R}^N} \ln |\xi|^2 |\mathcal{F}(u)|^2 ~d\xi < \infty \quad \text{and} \quad \int_{\mathbb{R}^N} u^2 \ln (e+|x|^2) ~dx < \infty.
\end{equation}
Then, we have 
    \begin{equation}\label{Log-Choquard-ineq-new}
 \begin{split}
       B_{N,\mathcal{L}} \|u\|_2^2 +  \frac{1}{\|u\|_2^2} \int_{\mathbb{R}^N} G_{\ln}(u) & u^2(x) ~dx ~dy - \frac{4}{N}  \int_{\mathbb{R}^N} u^2(x) \ln \frac{|u(x)|}{\|u\|_2}~dx \leq \int_{\mathbb{R}^N} \ln |\xi|^2 |\mathcal{F}(u)|^2 ~dx
    \end{split}
    \end{equation}
    where 
\begin{equation}\label{def:constant}
B_{N,\mathcal{L}} = \left[2 \frac{B_{N,1}}{B_{N,0}} + \frac{4}{N} \ln (B_{N,0})\right], \ B_{N,1} = B_{N,0} \left[ \ln 2 + 2 \psi\left(\frac{N}{2}\right) - \psi(N)\right] \ \text{and} \ B_{N,0}= \left(\frac{\Gamma(N)}{\pi^{\frac{N}{2}}\Gamma\left(\frac{N}{2}\right)}\right)^\frac{1}{2}.
\end{equation}
Moreover, the extremal functions of \eqref{Log-Choquard-ineq-new} are given up to conformal automorphism by $u(x)= A \left(1+ |x|^2\right)^{\frac{-N}{2}}$ for any $A  \in \mathbb{R}^+.$
\end{proposition}

Inequality \eqref{Log-Choquard-ineq-new} identifies the sharp scaling balance between the logarithmic diffusion, the convolution interaction, and the entropy correction, which further motivates the study of the problem in \eqref{main-problem}. Moreover, we do not identify \eqref{main-problem} with the Euler–Lagrange equation of this inequality. In particular,
knowing the extremals of the inequality does not by itself classify all positive solutions of the equation. Establishing that every such solution belongs to the bubble family is a separate rigidity problem. 

This places problem \eqref{main-problem} at the interface of two lines of research that, to the best of our knowledge, have not yet been brought together: the theory of the logarithmic Laplacian on one hand, and critical logarithmic Choquard-type problems on the other. 

As it is shown in \cite[Propositon 1.3]{Chen-Weth_2019}, the natural domain of pointwise definition of $L_{\Delta}$ is the set of uniformly Dini continuous functions $u$ in $\mathbb{R}^N$ with the integral restriction
\[
\|u\|_{L_0^1(\mathbb{R}^N)} :=
\int_{\mathbb{R}^N}
\frac{|u(x)|}{1+|x|^N}\,dx < \infty .
\]
Precisely, by denoting the modulus of continuity of $u$ at $x\in\mathbb{R}^N$ by
\[
\omega_{u,x}:[0,1]\to[0,\infty),
\qquad
\omega_{u,x}(r)=\sup_{y\in\mathbb{R}^N,\,|y-x|\le r}|u(y)-u(x)|,
\]
a function $u$ is said to be Dini continuous at $x$ if
\[
\int_0^1 \frac{\omega_{u,x}(r)}{r}\,dr<+\infty .
\]
Let $D_0(\mathbb{R}^N)$ be the set of all functions $u$ having the Dini continuity at any $x\in\mathbb{R}^N$ and denote
\[
\mathbb{H}_{\ln}(\mathbb{R}^N):= \left\{u \in L^2(\mathbb{R}^N): \int_{\mathbb{R}^N} u^2 \ln (e+|x|^2) ~dx + \left|\int_{\mathbb{R}^N} u^2 \ln |u| ~dx \right| < \infty\right\}.
\]
Before stating the main result, we say a function $u$ is said to be a (classical) solution of \eqref{main-problem} if
\[
u\in L_0^1(\mathbb{R}^N)\cap D_0(\mathbb{R}^N) \cap \mathbb{H}_{\ln}(\mathbb{R}^N)
\]
satisfies the equation \eqref{main-problem} pointwisely.
\newline
Our first result is the existence of Talenti bubbles type solutions:
\begin{theorem}\label{thm:bubble-solution}
    Let $\sigma \in \mathbb{R}$, $t>0$, $x_0 \in \mathbb{R}^N$, and 
    \begin{equation}\label{extremals}
         u_{\sigma,t}(x) := e^{\frac{N}{4}\left(\sigma - B_{N, \mathcal{L}}\right)} B_{N,0} \left(\frac{t}{t^2 + |x-x_0|^2}\right)^\frac{N}{2} \quad \text{for all} \ x \in \mathbb{R}^N 
    \end{equation}
    where the constants $B_{N,0}$ and $B_{N, \mathcal{L}}$ are defined in \eqref{def:constant}. Then, for any $\sigma \in \mathbb{R}$, $u_{\sigma,t}$ is a (classical) solution of 
\begin{equation}\label{main:prob}
\mathcal{L}_\Delta u= \sigma u + \frac{1}{\|u\|_2^2} \left[G_{\ln}(u) - \frac{4}{N}\left(\int_{\mathbb{R}^N} u^2 \ln u ~dx\right) \right] u \quad \text{in }\mathbb{R}^N.
    \end{equation}
\end{theorem}

To construct bubble type solution as in \eqref{extremals}, we consider the following normalized problem with $\|u\|_2 =1$ and $\sigma= B_{N,\mathcal{L}}$:
\begin{equation}\label{prob:log-choq}
 \mathcal{L}_\Delta u = B_{N,\mathcal{L}} u + \left[G_{\ln}(u)  - \frac{4}{N}\left(\int_{\mathbb{R}^N} u^2 \ln u ~dx\right)\right] u \quad \text{in} \ \mathbb{R}^N.
\end{equation}
The solutions are obtained as the limiting profile of solutions to the following critical fractional Choquard problem:
\begin{equation}\label{fractional-Choquard-prob}
(-\Delta)^s w_{s,t} = (|x|^{-4s} \ast w_{s,t}^2) w_{s,t} \quad \text{in} \ \mathbb{R}^N.
\end{equation}
By \cite[Theorem 2.15]{Avenia-Siciliano-Squassina_2015}, we know that every fixed-sign solution of \eqref{fractional-Choquard-prob} is of the form
\begin{equation}\label{talenti-def-funct}
w_{s,t} = c U_{s,t}, \quad \text{where} \quad U_{s,t}(x) = \left(\frac{t}{t^2+|x-x_0|^2}\right)^{\frac{N-2s}{2}}, \qquad x \in \mathbb{R}^N,
\end{equation}
for $s \in \left[0,\frac{1}{4}\right)$, and some $t>0$, $x_0 \in \mathbb{R}^N$, and $c>0$. Moreover, equation \eqref{fractional-Choquard-prob} is the Euler--Lagrange equation associated with the Hardy--Littlewood--Sobolev inequality; see \cite[Theorem 3.1]{Lieb-1983}. The first step of the proof is to determine the sharp constant $c=B_{N,s}$ explicitly (see \eqref{frac-choquard-constant}) and to derive its asymptotic behavior as $s \to 0^+$. These asymptotic estimates, together with the limiting profile of the fractional Talenti bubbles $U_{s,t}$, play a crucial role in deriving the limiting profile of the family $w_{s,t}=B_{N,s}U_{s,t}$. Next, we derive the first-order expansions of both sides of \eqref{fractional-Choquard-prob}. For the operator term, we employ the first-order expansion of $(-\Delta)^s$ in terms of $\mathcal{L}_\Delta$ established in \cite[Theorem 1.1]{Chen-Weth_2019}. For the nonlocal convolution term, we apply Taylor expansions separately to the singular kernel $|x|^{-4s}$ and to the quadratic term $w_{s,t}^2$, and then analyze the interaction between these expansions. This interaction produces nonlinear terms that are not invariant with respect to the scaling parameter $t$.

Finally, by combining the asymptotic behavior of the constant $B_{N,s}$ with the norm estimates for the fractional Talenti bubbles $U_{s,t}$, we pass to the limit as $s\to0^+$ and conclude that the limiting profile of the fractional Talenti bubbles is a solution of the logarithmic Choquard problem \eqref{prob:log-choq}. Moreover, by exploiting the non-homogeneity of the correction term appearing in \eqref{prob:log-choq}, we show that the family of solutions defined in \eqref{extremals} satisfies the logarithmic Choquard problem \eqref{prob:log-choq} for any $\sigma \in \mathbb{R}.$

Next, we investigate the uniqueness of solutions to problem \eqref{prob:log-choq} within the class of limiting profiles of the fractional Talenti bubbles defined in \eqref{extremals}. To this end, we employ the method of moving planes together with techniques based on the Kelvin transform. However, extending these methods to the present setting involves several substantial difficulties.


The main difficulty arises from the doubly nonlocal structure of the nonlinearity, which involves a sign-changing and scaling non-invariant logarithmic kernel. Consequently, the recent approach developed in \cite{Chen-Zhou_2025} cannot be applied directly. Indeed, the arguments therein rely crucially on decomposing the logarithmic nonlinearity into its positive and negative parts, a strategy that is no longer applicable in the presence of the nonlocal convolution term.

Moreover, the approach of \cite{Chen-Li-Ou_2006} based on the associated integral equation is also unavailable in our setting. This is due to the fact that the fundamental solution of the logarithmic Laplacian $\mathcal{L}_\Delta$ is neither positive nor monotonically decreasing, together with the lack of Sobolev embedding theorems associated with the logarithmic Laplacian in $\mathbb{R}^N$ (see \cite{Chen-Veron-2024, Felmer-Yang-2014}). To overcome these difficulties, we develop a method of moving planes tailored to the logarithmic Laplacian and the doubly nonlocal structure of the equation, while carefully accounting for the nonlocal effects induced by the logarithmic convolution terms.

\begin{theorem}\label{thm-radi-mono-prop}
    Let $N \geq 1$, $\sigma \in \mathbb{R}$, $u \in L^1_0(\mathbb{R}^N) \cap \mathbb{H}_{\ln}(\mathbb{R}^N) \cap \mathcal{D}^0(\mathbb{R}^N)$ be a positive solution of \eqref{main:prob}.
Then, $u$ is radially symmetric and monotonically decreasing about some point in $\mathbb{R}^N$. Furthermore, there exists $u_\infty \in  (0, \infty)$ such that
\begin{equation}\label{asym-beha}
        \lim_{|x| \to \infty} |x|^N u(x) = u_\infty
\end{equation}
and 
\begin{equation}\label{solu:upperbound:estimate}
\|u\|_2 \leq e^{\frac{N}{4}\left(\sigma- B_{N, \mathcal{L}}\right)}. 
\end{equation}
\end{theorem}
\begin{remark}
    Note that for any $u \in \mathbb{H}_{\ln}(\mathbb{R}^N)$, both the integrals terms $G_{\ln}(u)$ and $\int_{\mathbb{R}^N} G_{\ln}(u) u^2 ~dx$ in \eqref{main:prob} are well defined due to the following inequality
    \[
    \ln |x-y| \leq \frac{1}{2} \left(\ln 2 + \ln (1+|x|^2) + \ln (1+ |y|^2)\right) \ \text{for any $x, y \in \mathbb{R}^N$}
    \]
and Beckner's entropy inequality, given in Lemma \ref{Beckner-ineq}.
\end{remark}

Theorem \ref{thm-radi-mono-prop} extends the moving-plane argument developed by Chen and Zhou \cite{Chen-Zhou_2025} from the local logarithmic nonlinearity to the present nonlocal Choquard-type problem involving the logarithmic potential
$
G_{\ln}(u)$
together with a global entropy correction term. The main hurdle arises from the nonlocal nature of the reaction term. In the local setting considered in \cite{Chen-Zhou_2025}, the difference $w_\lambda$ of two Kelvin-reflected solutions solves a linear equation whose coefficient is an explicit, sign-definite pointwise quantity obtained from the monotonicity of $t \mapsto t\ln t$. In contrast, for the present problem, the reflected equation (see \eqref{main-Kelvin-differ-0}) contains the difference of two logarithmic potentials, which acts as a nonlocal integral term whose kernel depends on the plane's position (see $I_\lambda(x)$ in \eqref{differ-rewrite}).

To overcome this difficulty, we adapt the techniques of \cite{Cingolani-Weth-2016, Cingolani-Weth-2022}, originally developed in the local setting. We first establish a Kelvin-invariance identity for the logarithmic potential, which follows from the reflection symmetry of the logarithmic kernel. We then derive a decay estimate for $G_{\ln}(u_1)$ at infinity (see \eqref{upper-est-I1}) by splitting the convolution into near- and far-field contributions and using the decay rate of the Kelvin transform. Finally, we obtain a quantitative bound (see \eqref{est-norm-est-I2}) showing that the nonlocal reaction term $I_\lambda$ is controlled by the $L^2$-norm of the negative part of the difference of two Kelvin-reflected solutions, with a constant $C_\lambda \to 0$. This last estimate is the technical core of the paper: it fills the role as the pointwise sign condition on $V_\lambda,W_\lambda$ and narrow region maximum principles in \cite{Chen-Zhou_2025}, but is obtained through a potential theoretic mechanism. Furthermore, it allows us to complete the moving-plane procedure showing that the optimal position $\Lambda$ (see \eqref{def-Lambda}) of the plane coincides with the symmetry hyperplane. Moreover, a direct application of the sharp inequality in Proposition \ref{new:Log-Choq:ineq} gives the upper bound estimate in \eqref{solu:upperbound:estimate}.

Next, by using the geometric arguments adapted from \cite{Chen-Li-Ou_2006} (see also \cite{Chen-Zhou_2025}), radial symmetry and monotonicity obtained in Theorem~\ref{thm-radi-mono-prop} and properties of the Kelvin transform, we show that any positive classical solution attain the bound in \eqref{solu:upperbound:estimate} and derive the uniqueness of positive solutions. 

\begin{theorem}\label{thm-uniqueness}
Let $u_{\sigma, t}$ be the function defined in \eqref{extremals}. If $u$ is a positive solution of \eqref{main:prob} with $\sigma \in \mathbb{R}$, then $\|u\|_{2}=e^{\frac{N}{4} (\sigma-B_{N, \mathcal{L}})},$ and there exist $t>0$ and $\tilde{x}\in\mathbb{R}^N$ such that
\[
u(x)=u_{\sigma,t}(x-\tilde{x}), \qquad x\in\mathbb{R}^N.
\]
\end{theorem}
\noindent \textbf{Outline of the paper:} In Section \ref{existence}, we derive the sharp explicit constants associated with the fractional Choquard equation and show that the limiting profiles of both the sharp constants and the fractional Talenti bubbles solve problem \eqref{prob:log-choq}. In Section \ref{radial-mono}, we establish the invariance properties of the modified problem \eqref{main:prob}, which reduces to the original problem \eqref{prob:log-choq} under a suitable normalization. We also prove the radial symmetry and monotonicity of solutions to \eqref{main:prob}. In Section \ref{uniqueness}, we establish the uniqueness of solutions corresponding to the limiting profile of the fractional Talenti bubbles. Finally, in Appendix \ref{appendix}, we derive the asymptotic behavior of the constants involved in the proofs and establish several auxiliary norm estimates.
\section{Existence of Talenti bubbles type solutions}\label{existence}
In this section, we first establish the proof of Logarithmic Hardy–Littlewood–Sobolev entropy inequality in \eqref{Log-Choquard-ineq-new}, which plays a crucial role in the study of the logarithmic Choquard problem \eqref{main:prob}. Next, we derive the sharp explicit constants appearing in the solutions of the fractional Choquard problem. By establishing their asymptotic behavior as $s \to 0^+$, we show that the limiting profiles of the constants and fractional Talenti bubbles $U_{s,t}$ yield solutions to the logarithmic Choquard problem \eqref{prob:log-choq}.
\subsection{Logarithmic Hardy-Littlewood-Sobolev entropy inequality}
\textbf{Proof of Proposition \ref{new:Log-Choq:ineq}:} It is easy to see that
\begin{equation}\label{rela:Fourier}
    \mathcal{F}_0(u)(\xi) = (2 \pi)^\frac{N}{2} \mathcal{F}(u)(-2 \pi \xi) \ \text{for} \ \xi \in \mathbb{R}^N
\end{equation}
where $\mathcal{F}_0(u)(\xi):= \int_{\mathbb{R}^N} e^{2 \pi \iota x \xi} u(x) ~dx$. 
Now, by using the Beckner's entropy inequality in \eqref{LOg-HLS-ineq-2} and Pitt's inequality by Beckner in \eqref{Pitts:ineq} for $\frac{u}{\|u\|_2}$ with $u$ satisfying \eqref{reg:weak-solu}, we obtain 
    \begin{equation}\label{LOg-HLS-ineq-new-1}
 \begin{split}
        \frac{1}{\|u\|_2^2} & \int_{\mathbb{R}^N}  \left(\ln \left(\frac{1}{|x|^4}\right) \ast u^2\right) u^2(x) ~dx ~dy - \frac{4}{N}  \int_{\mathbb{R}^N} u^2(x) \ln |u(x)|~dx \\
        & \leq \frac{2C_N}{N} \|u\|_2^2 + \frac{4}{N}  \int_{\mathbb{R}^N} u^2(x) \ln |u(x)|~dx - \frac{8}{N} \|u\|_2^2 \ln \|u\|_2\\
        & \leq \left(\frac{2C_N}{N} - \frac{4 G_N}{N} \right) \|u\|_2^2 + \int_{\mathbb{R}^N} \ln |\xi|^2 |\mathcal{F}_0(u)|^2 ~dx - \frac{4}{N} \|u\|_2^2 \ln \|u\|_2.
    \end{split}
    \end{equation}
Next, using \eqref{rela:Fourier}, applying the change of variables in \eqref{LOg-HLS-ineq-new-1}, and invoking Plancherel's theorem, we obtain        \begin{equation*}\label{LOg-HLS-ineq-new-10}
 \begin{split}
        \frac{1}{\|u\|_2^2} & \int_{\mathbb{R}^N}  \left(\ln \left(\frac{1}{|x|^4}\right) \ast u^2\right) u^2(x) ~dx ~dy - \frac{4}{N}  \int_{\mathbb{R}^N} u^2(x) \ln |u(x)|~dx \\
        & \leq \left(\frac{2C_N}{N} - 2 \ln (2\pi) - \frac{4 G_N}{N} \right) \|u\|_2^2 + \int_{\mathbb{R}^N} \ln |\xi|^2 |\mathcal{F}(u)|^2 ~dx - \frac{4}{N} \|u\|_2^2 \ln \|u\|_2\\
        & = \int_{\mathbb{R}^N} \ln |\xi|^2 |\mathcal{F}(u)|^2 ~dx - \left(B_{N,\mathcal{L}} + \frac{4}{N} \ln \|u\|_2 \right) \|u\|_2^2
    \end{split}
    \end{equation*}
where we have used that $B_{N, \mathcal{L}} = \frac{4}{N} G_N + 2\ln (2\pi) -\frac{2C_N}{N}.$ Moreover, the characterization of the extremal functions follows from Lemmas \ref{lem:Pitt's ineq} and \ref{Beckner-ineq}. \qed
\subsection{Logarithmic Laplacian Choquard problem}

Next, we find the explicit constant $c$ in $w_{s,t}$ defined in \eqref{talenti-def-funct}.  This plays a crucial role in the deriving the asymptotics of the solution $w_{s,t}$ and further leads to the exact solution of the problem \eqref{prob:log-choq}. 
\begin{lemma}\label{crit-choq-solu}
Let $s \in (0, \frac{1}{4}).$ Then, the function $w_{s,t}(x) = B_{N,s} U_{s,t}$ satisfies \eqref{fractional-Choquard-prob} and 
    \begin{equation}\label{frac-choquard-constant}
     \|w_{s,t}\|_2 = (B_{N,s}) \|U_{s,t}\|_2 =(t^{2s} A_{N,s})^\frac{1}{2} \quad \text{with} \quad B_{N,s} = \left(\frac{2^{2s}\Gamma\left(\frac{N+2s}{2}\right) \Gamma\left(N-2s\right)}{\pi^{\frac{N}{2}} \Gamma\left(\frac{N-2s}{2}\right)\Gamma\left(\frac{N-4s}{2}\right)} \right)^\frac{1}{2}.
    \end{equation} 
\end{lemma}
\begin{proof}
It follows from \cite[Theorem 1.1]{Chen-Li-Ou_2006} and \cite[(1.7)]{Chen-Zhou_2025} that the function $v:\mathbb{R}^N \to \mathbb{R}$ defined by
\begin{equation}\label{Sob-critical-bubble-solu}
v(x) = b_{N,s} U_{s,t}(x), \qquad \text{where} \qquad
b_{N,s} = 2^{\frac{N-2s}{2}}
\left(\frac{\Gamma\left(\frac{N+2s}{2}\right)}
{\Gamma\left(\frac{N-2s}{2}\right)}\right)^\frac{N-2s}{4s},
\end{equation}
is a positive solution of
\[
(-\Delta)^s v = v^{2_s^\ast-1}
\qquad \text{in } \mathbb{R}^N \quad \ \text{where} \quad 2_s^\ast = \frac{2N}{N-2s}.
\]
Consequently, for every $x_0 \in \mathbb{R}^N$ and $t>0$, the function $U_{s,t}$ satisfies
\begin{equation}\label{fractional-Sobo-crit}
(-\Delta)^s U_{s,t} = A_{N,s} U_{s,t}^{2_s^\ast-1}
\quad \text{in } \mathbb{R}^N,
\qquad \text{where} \qquad
A_{N,s}
:=
\frac{2^{2s}\Gamma\left(\frac{N+2s}{2}\right)}
{\Gamma\left(\frac{N-2s}{2}\right)}.
\end{equation}
Now, by using \eqref{fractional-Choquard-prob}-\eqref{talenti-def-funct} and \eqref{fractional-Sobo-crit}, we have
   \[
   A_{N,s} U_{s,t}^{2_s^\ast-1} = (-\Delta)^s  U_{s,t} =  c^2 (|x|^{-4s} \ast U_{s,t}^2) U_{s,t}.
   \]
Muliplying the equation by $U_{s,t}$, integrating over $\mathbb{R}^N$ and using Theorem \ref{HLS-inequality}, we obtain
\begin{equation}\label{est:crit-choq-1}
    A_{N,s} \|U_{s,t}\|_{2_s^\ast}^{2_s^\ast} = c^2 \int_{\mathbb{R}^N} (|x|^{-4s} \ast U_{s,t}^2) U_{s,t}^2 ~dx = c^2 C_{N,s} \|U_{s,t}\|_{2_s^\ast}^4 \quad \Longrightarrow \quad c^2 = \frac{A_{N,s}}{C_{N,s} \|U_{s,t}\|_{2_s^\ast}^{4-2_s^\ast}}.
\end{equation}
Finally, by using \eqref{fractional-Sobo-crit}, \eqref{HLS-inequality}, and Lemma \ref{lem:norm-compu} in \eqref{est:crit-choq-1}, we obtain the required claim.
\end{proof}
\noindent 
\textbf{Proof of Theorem \ref{thm:bubble-solution}:}  From Lemma \ref{crit-choq-solu}, note that the function $w_{s,t} = B_{N,s} U_{s,t}$ is a solution of \eqref{fractional-Choquard-prob}. Fixing $t>0$, by Taylor's expansion Theorem, we have
    \begin{equation}\label{est:Taylor-1}
        U_{s,t} = U_{0,t} - \frac{2s}{N} U_{0,t} \ln U_{0,t} + s^2 \tilde{U}_{s, t} \quad \text{in} \ \mathbb{R}^N
    \end{equation}
    where $U_{s,t}$ is defined in \eqref{talenti-def-funct} and 
    \[
    \tilde{U}_{s,t} := \frac{2}{N^2}U_{s',t} \ (\ln U_{0,t})^2, \ \text{for some} \ s' \in [0,s]
    \]
and satisfies 
    \begin{equation}\label{est:regularity}
        \limsup_{s \to 0^+} \left[ \sup_{x \in \mathbb{R}^N} \left(|\tilde{U}_{s,t}(x)| + |\nabla \tilde{U}_{s,t}(x)|\right) + \||\tilde{U}_{s,t}| + |\nabla \tilde{U}_{s,t}|\|_{q} \right] < +\infty, \quad \text{for any} \ q>1.
    \end{equation}
Now, by using Lemma \ref{lem:constants} and \eqref{est:Taylor-1}, we obtain   
\begin{equation}\label{est:solu-taylor}
    \begin{split}
w_{s,t} & = \left( B_{N,0} + s B_{N,1} + o(s) \right) \left(U_{0,t} - \frac{2s}{N} U_{0,t} \ln U_{0,t} + s^2 \tilde{U}_{s, t}\right)\\
& = B_{N,0} U_{0,t} + s \left(B_{N,1} U_{0,t} - \frac{2}{N} B_{N,0} U_{0,t} \ln U_{0,t} \right) + s^2 \hat{u}_{s,t} 
\end{split}
\end{equation}
where $\hat{u}_{s,t}$ satisfies property as \eqref{est:regularity}. Then, by using \cite[Theorem 1.1]{Chen-Weth_2019} and applying \cite[Theorem 3.1]{Chen-Zhou_2025} with $u_s = w_{s,t}$, $u_0 =B_{N,0} U_{0,t}$ and $u_1 = B_{N,1} U_{0,t} - \frac{2}{N} B_{N,0} U_{0,t} \ln U_{0,t} $, uniformly for any $x \in \mathbb{R}^N$, we obtain
\begin{equation}\label{est:fractional-solu}
    (-\Delta)^s w_{s,t}(x) = B_{N,0} U_{0,t} + s \left(B_{N,1} U_{0,t} - \frac{2}{N} B_{N,0} U_{0,t} \ln U_{0,t} + B_{N,0} \mathcal{L}_\Delta U_{0,t} \right) + o(s).
\end{equation}
From \eqref{est:solu-taylor}, we have
\begin{equation}\label{est:solu-square}
      w_{s,t}^2 = B_{N,0}^2 U_{0,t}^2 + 2 s \left[ B_{N,0} B_{N,1} U_{0,t}^2 - \frac{2}{N} B_{N,0}^2 U_{0,t}^2 \ln U_{0,t} \right] + o(s).
\end{equation}
By Taylor's theorem and using \eqref{est:solu-square}, we get
\begin{equation}\label{est:convolution-term}
    \begin{split}
        |x|^{-4s} \ast w_{s,t}^2 & = \left(1 + s \ln \left(\frac{1}{|x|^4}\right) + o(s)\right) \ast \left(B_{N,0}^2 U_{0,t}^2 + 2 s \left[ B_{N,0} B_{N,1} U_{0,t}^2 - \frac{2}{N} B_{N,0}^2 U_{0,t}^2 \ln U_{0,t} \right] + o(s)\right)\\
        & = B_{N,0}^2 \|U_{0,t}\|^2_2 \\
        & \quad + s \left[ B_{N,0}^2 \ln \left(\frac{1}{|x|^4}\right) \ast U_{0,t}^2 + 2 B_{N,0} B_{N,1} \|U_{0,t}\|^2_2 - \frac{4 B_{N,0}^2}{N} \left(\int_{\mathbb{R}^N} U_{0,t}^2 \ln U_{0,t} ~dy\right) \right] + o(s).
    \end{split}
\end{equation}
Now, by multiplying \eqref{est:convolution-term} by $w_{s,t}$ and using \eqref{est:solu-taylor}, we obtain
\begin{equation}\label{product-terms}
    \begin{split}
    \left(|x|^{-4s} \ast w_{s,t}^2 \right) w_{s,t} & = B_{N,0}^2 \|U_{0,t}\|^2_2 \left(B_{N,0} U_{0,t} + s \left(B_{N,1} U_{0,t} - \frac{2}{N} B_{N,0} U_{0,t} \ln U_{0,t} \right) + s^2 \hat{u}_{s,t} \right) \\
        & \quad + s \left[ B_{N,0}^2 \ln \left(\frac{1}{|x|^4}\right) \ast U_{0,t}^2 + 2 B_{N,0} B_{N,1} \|U_{0,t}\|^2_2 - \frac{4 B_{N,0}^2}{N} \left(\int_{\mathbb{R}^N} U_{0,t}^2 \ln U_{0,t} ~dy\right) \right] \\
        & \qquad \qquad \times \left(B_{N,0} U_{0,t} + s \left(B_{N,1} U_{0,t} - \frac{2}{N} B_{N,0} U_{0,t} \ln U_{0,t} \right) + s^2 \hat{u}_{s,t} \right) + o(s)\\
        & = B_{N,0}^3 \|U_{0,t}\|^2_2 U_{0,t} + s B_{N,0}^2 \|U_{0,t}\|^2_2 \left[ B_{N,1} U_{0,t} - \frac{2}{N} B_{N,0} U_{0,t} \ln U_{0,t}\right] \\
         & \qquad + s \left[ B_{N,0}^3 \left(\ln \left(\frac{1}{|x|^4}\right) \ast U_{0,t}^2\right) U_{0,t} \right] \\
        & \qquad + s  \left[2 B_{N,0}^2 B_{N,1} \|U_{0,t}\|^2_2 U_{0,t} - \frac{4 B_{N,0}^3}{N} \left(\int_{\mathbb{R}^N} U_{0,t}^2 \ln U_{0,t} ~dy\right) U_{0,t}\right] + o(s).\\
\end{split}
\end{equation}
Further, by Lemma \ref{constants-value} in \eqref{product-terms}, we obtain
\begin{equation}\label{est:convolution-nonlinearities}
    \begin{split}
        & \left(|x|^{-4s} \ast w_{s,t}^2 \right) w_{s,t} \\
        & \quad = B_{N,0} U_{0,t} + s \left[ B_{N,1} U_{0,t} - \frac{2}{N} B_{N,0} U_{0,t} \ln U_{0,t}\right] + s  \left[2 \frac{B_{N,1}}{B_{N,0}}  + \frac{4}{N} \ln (B_{N,0})\right] (B_{N,0}  U_{0,t})   \\
        & \quad \qquad + s \left[ \left(\ln \left(\frac{1}{|x|^4}\right) \ast (B_{N,0} U_{0,t})^2 - \frac{4}{N} \ast (B_{N,0} U_{0,t})^2 \ln (B_{N,0} U_{0,t})) \right) B_{N,0} U_{0,t} \right] + o(s).
    \end{split}
\end{equation}
Now, by using \eqref{fractional-Choquard-prob} and combining \eqref{est:fractional-solu} and \eqref{est:convolution-nonlinearities}, dividing  by $s$ and  passing limit $s \to 0^+$, we obtain, the function $u_{\sigma,t}$ defined in \eqref{extremals}, is a solution of the problem \eqref{main:prob} with $\sigma = B_{N, \mathcal{L}}.$ Finally, set $v = \alpha u_{\beta}$ where $u_\beta$ is the solution of \eqref{main:prob} for some $\sigma=\beta \in \mathbb{R}$ . Then, for any $x \in \mathbb{R}^N$, we obtain
\begin{equation}\label{main:prob-modi}
\begin{split}
\mathcal{L}_\Delta v = \alpha \mathcal{L}_\Delta u_\beta &= \beta (\alpha u_\beta) + \frac{4 \ln |\alpha|}{N} (\alpha u_\beta) + \frac{1}{\|\alpha u_\beta\|_2^2} \left[G_{\ln}(\alpha u_\beta) -  \frac{4}{N}\left(\int_{\mathbb{R}^N} (\alpha u_\beta)^2 \ln |\alpha u_\beta| ~dx\right) \right] \alpha u_\beta \\
& = \left(\beta + \frac{4 \ln |\alpha|}{N}\right) v + \frac{1}{\|v\|_2^2} \left[G_{\ln}(v) -  \frac{4}{N}\left(\int_{\mathbb{R}^N} v^2 \ln |v| ~dx\right) \right] v.
\end{split}
    \end{equation}
Therefore, for any $\sigma \in \mathbb{R}$, there exists a $\beta \in \mathbb{R}$ such that the function $v := \alpha(\beta) u_\beta$ is a solution of the problem \eqref{main:prob}
with $\alpha(\beta) = e^{\frac{N}{4}\left(\sigma - \beta\right)}.$ Finally, by taking $\beta= B_{N,\mathcal{L}}$, we obtain the required claim. 
\qed

\section{Radial Symmetry}\label{radial-mono}
In this section, we first derive the translation, dilation and Kelvin transform invariant properties of the modified problem \eqref{main:prob}. After that, by applying the moving plane method, we derive the radial symmetry and monotonicity propertes of the classical solution of the problem \eqref{main:prob}.
\subsection{Invariance properties}
We recall the following identity satisfied by the Euclidean norm:
\begin{equation}\label{pointwise:eq}
    \left|\frac{X}{|X|^2}- \frac{Y}{|Y|^2}\right| = \frac{|X-Y|}{|X| |Y|}  \quad \ \text{for any} \ X, Y \in \mathbb{R}^N \setminus \{0\}.
\end{equation}
Denote for a fixed $\bar{x}\in \mathbb{R}^N$
\[
x^{\ast,r} = \frac{r^2 (x-\tilde{x})}{|x-\tilde{x}|^2} + \tilde{x}, \quad \text{for} \ x \in \mathbb{R}^N \setminus \{\tilde{x}\}.
\]
\begin{proposition}
    The following holds:
    \begin{enumerate}
        \item[$(i)$] \begin{equation}\label{est:iden-1}
            |x-y| = \frac{1}{r^2} |x^{\ast,r} - y^{\ast,r} | |y-\tilde{x}||x-\tilde{x}| =  \frac{|x^{\ast,r} - y^{\ast,r}| |x-\tilde{x}|}{|y^{\ast,r} -\tilde{x}|}, \quad \text{where} \quad x, y \in \mathbb{R}^N \setminus \{\tilde{x}\}.
        \end{equation}
        \item[$(ii)$] 
        \begin{equation}\label{est:iden-2}
          |x-y| = r^2 \frac{|\xi-\eta|}{|\xi-\tilde{x}||\eta-\tilde{x}|},  \quad \text{where} \quad \xi= x^{\ast,r},  \eta= y^{\ast,r} \ \text{and}\ x, y \in \mathbb{R}^N \setminus \{\tilde{x}\}.  
        \end{equation} 
    \end{enumerate}
\end{proposition}
\begin{proof} By taking $X= x-\tilde{x}$ and $Y= y-\tilde{x}$ in \eqref{pointwise:eq}, we get
\begin{equation}\label{est-Kel-eq-1}
        |x^{\ast,r} - y^{\ast,r} | = r^2 \left|\frac{(x-\tilde{x})}{|x-\tilde{x}|^2}- \frac{(y-\tilde{x})}{|y-\tilde{x}|^2} \right| = r^2 \frac{|x-y|}{|x-\tilde{x}| |y-\tilde{x}|}.
\end{equation}
Hence, the claim in $(i)$ follows. Note that
\begin{equation}\label{est-Kel-eq-2}
    |\xi-\tilde{x}| = \frac{r^2}{|x-\tilde{x}|}, \quad  |\eta-\tilde{x}| = \frac{r^2}{|y-\tilde{x}|}.
\end{equation}
Using \eqref{est-Kel-eq-1} and \eqref{est-Kel-eq-2}, we obtain    
\[
|\xi-\eta| = |x^{\ast,r} - y^{\ast,r} | = r^2 \frac{|x-y|}{|x-\tilde{x}| |y-\tilde{x}|} = \frac{1}{r^2} |x-y| |\xi-\tilde{x}| |\eta-\tilde{x}|.
\]
Hence, the claim in $(ii)$ holds.
\end{proof}
\begin{proposition}\label{pro:invariance}
    Let $\sigma \in \mathbb{R}$ and $u \in \mathcal{D}^0(\mathbb{R}^N) \cap L_0^1(\mathbb{R}^N) \cap \mathbb{H}_{\ln}(\mathbb{R}^N)$ be a solution of 
    \begin{equation}\label{parameterized:prob}
     \begin{cases}
\mathcal{L}_\Delta u= \sigma u 
+  \frac{1}{\|u\|_2^2} \left[G_{\ln}(u) - \frac{4}{N} \int_{\mathbb{R}^N} u^2 \ln |u| ~dx \right] u & \text{in }\mathbb{R}^N,\\
u\ge 0 & \text{in }\mathbb{R}^N.
\end{cases}
    \end{equation}
\begin{enumerate}
    \item[$(i)$] Given $\lambda >0$ and $x_0 \in \mathbb{R}^N$, let $$w_\lambda(x) = \lambda u(x+x_0), \quad x \in \mathbb{R}^N.$$ Then, $w_\lambda$ satisfies  
        \[
\begin{cases}
\mathcal{L}_\Delta w_\lambda = \sigma w_\lambda + \frac{1}{\|w_\lambda\|_2^2}\left[ G_{\ln}(w_\lambda) - \frac{4}{N}\left(\int_{\mathbb{R}^N} w_\lambda^2 \ln w_\lambda~dx\right) \right] w_\lambda - \frac{4 \ln \lambda}{N} w_\lambda(x) & \text{in }\mathbb{R}^N,\\
w_\lambda \ge 0 & \text{in }\mathbb{R}^N.
\end{cases}
\]
\item[$(ii)$] Given $\lambda >0$, let $$v_\lambda(x) = \lambda^{-\frac{N}{2}} u(x/\lambda), \quad x \in \mathbb{R}^N.$$ Then, $v_\lambda$ satisfies \eqref{parameterized:prob}. 
In addition, $\|u\|_2 = \|v_\lambda\|_2$ for every $\lambda >0.$
\item[$(iii)$] Let $\tilde{x} \in \mathbb{R}^N$, $r>0$ and 
\begin{equation}\label{def:Kelvin-trans}
    u^{\#}_r (x) = \left(\frac{r}{|x-\tilde{x}|}\right)^N u(x^{\ast,r}).
\end{equation}
Then, $u^{\#}_r$ satisfies
\[
\mathcal{L}_\Delta u^{\#}_r = \sigma u^{\#}_r  + \frac{1}{\|u^{\#}_r\|_2^2}\left[ G_{\ln}(u^{\#}_r) - \frac{4}{N} \int_{\mathbb{R}^N} (u^{\#}_r)^2 \ln |u^{\#}_r| ~dx  \right] u^{\#}_r \quad \text{in} \ \mathbb{R}^N \setminus \{\tilde{x}\}.
\]
\item[$(iv)$] Let $Q\in O(N)$ be an orthogonal matrix and define
$$
v(x):=u(Qx),\qquad x\in\mathbb{R}^N.
$$
Then $v$ satisfies \eqref{parameterized:prob}. In particular, the equation is invariant under rotations and reflections.
\end{enumerate}
\end{proposition}
\begin{proof}
    Let $u$ be a classical solution of \eqref{parameterized:prob} and $\sigma \in \mathbb{R}$. Note that
    \[
    (L_\Delta w_\lambda) (x) = \lambda (L_\Delta u) (x+x_0), \quad G_{\ln} (w_\lambda)(x) = \lambda^2 G_{\ln} (u)(x+x_0),
    \]
    and 
 \[
    \int_{\mathbb{R}^N} w_\lambda^2 \ln w_\lambda~dx = \lambda^2 \int_{\mathbb{R}^N} u^2 \ln u ~dx + \lambda^2 \ln \lambda \|u\|_2^2. 
    \]
    This further gives that the function $w_\lambda$ satsifes
    \[
    \begin{split}
        \mathcal{L}_\Delta w_\lambda (x) & = \lambda \mathcal{L}_\Delta u (x+x_0)= \lambda \sigma u (x+x_0) + \frac{\lambda}{\|u\|_2^2} \left[G_{\ln}(u) (x+x_0) -  \frac{4}{N} \int_{\mathbb{R}^N} u^2 \ln u ~dx\right] u (x+x_0)\\
        & = \sigma w_\lambda(x) + \frac{1}{\|w_\lambda\|_2^2} \left[ G_{\ln}(w_\lambda)(x) -  \frac{4}{N} \int_{\mathbb{R}^N} w_\lambda^2 \ln w_\lambda ~dx\right] w_\lambda (x) - \frac{4 \ln \lambda}{N} w_\lambda (x) \quad \text{for} \ x \in \mathbb{R}^N.
    \end{split}
    \]
Hence, the proof $(i)$. Next, we show $(ii)$. By straightforward computations, we have $\|u\|_2 = \|v_\lambda\|_2$ for every $\lambda >0.$ Note that, we have
\begin{align*}
\mathcal{F}(L_\Delta v_\lambda)(\xi) & = \bigl(2\ln |\xi|\bigr)\widehat{v_\lambda}(\xi) = \bigl(2\ln |\xi|\bigr)\lambda^{\frac{N}{2}}\widehat{u}(\lambda\xi) = \bigl(2\ln |\lambda\xi|\bigr)\lambda^{\frac{N}{2}}\widehat{u}(\lambda \xi)
   - \bigl(2\ln \lambda\bigr)\lambda^{\frac{N}{2}}\widehat{u}(\lambda \xi)
\end{align*}
where we used the fact that $\widehat{v_\lambda}(\xi) = \lambda^{\frac{N}{2}}\widehat{u}(\lambda\xi)$. Then, by applying the inverse Fourier transform, we obtain
\begin{equation}\label{est:inv-1}
    \mathcal{L}_\Delta v_\lambda (x) = \lambda^{-\frac{N}{2}} (\mathcal{L}_\Delta u)\left(\frac{x}{\lambda}\right) - 2 \ln \lambda ~ v_\lambda, 
\end{equation}
\begin{equation}\label{est:inv-1-2}
    \begin{split}
G_{\ln}(v_\lambda)(x) & = \int_{\mathbb{R}^N} \ln\left(\frac{1}{|x-y|^4}\right) v_\lambda^2 (y) ~dy = \lambda^{-N} \int_{\mathbb{R}^N} \ln\left(\frac{1}{|x-y|^4}\right) u^2 \left(\frac{y}{\lambda}\right) ~dy\\
& = \lambda^{-N} \int_{\mathbb{R}^N} \ln\left(\frac{1}{\left|\frac{x}{\lambda}- \frac{y}{\lambda}\right|^4}\right) u^2 \left(\frac{y}{\lambda}\right) ~dy + \lambda^{-N} \int_{\mathbb{R}^N} \ln\left(\frac{1}{\lambda^4}\right) u^2 \left(\frac{y}{\lambda}\right) ~dy \\
& = G_{\ln}(u)\left(\frac{x}{\lambda}\right) - 4 \ln \lambda \|v_\lambda\|_2^2
\end{split}
\end{equation}
and
\begin{equation}\label{est:inv-2}
\begin{split}
\int_{\mathbb{R}^N} v_\lambda^2 \ln v_\lambda ~dx &= \lambda^{-N} \int_{\mathbb{R}^N}  u^2\left(\frac{x}{\lambda}\right) \ln \left(u\left(\frac{x}{\lambda}\right)\right) ~dx - \frac{N}{2} \ln \lambda \|u\|_2^2 = \int_{\mathbb{R}^N} u^2 \ln u ~dx - \frac{N}{2} \ln \lambda \|v_\lambda\|_2^2
\end{split}
\end{equation}
where we have used the fact that $\|u\|_2 = \|v_\lambda\|_2$ for every $\lambda >0.$ Now, by combining estimates in \eqref{est:inv-1}, \eqref{est:inv-1-2} and \eqref{est:inv-2}, we obtain
\begin{equation*}
    \begin{split}
        \mathcal{L}_\Delta v_\lambda (x) &= \sigma v_\lambda + \frac{1}{\|u\|_2^2}\left[G_{\ln}(u)\left(\frac{x}{\lambda}\right) - \frac{4}{N}\left(\int_{\mathbb{R}^N} u^2 \ln u ~dx\right)\right] v_\lambda - 2 \ln \lambda ~ v_\lambda\\
        & = \sigma v_\lambda  - 2 \ln \lambda ~ v_\lambda\\
        & \qquad + \frac{1}{\|v_\lambda\|_2^2}\left[G_{\ln}(v_\lambda)(x) + 4 \ln \lambda \|v_\lambda\|_2^2 - \frac{4}{N}\left(\int_{\mathbb{R}^N} G_{\ln}(v_\lambda) v_\lambda^2 ~dx\right) - 2 \ln \lambda \|v_\lambda\|_2^2 \right] v_\lambda\\
        & = \sigma v_\lambda + \frac{1}{\|v_\lambda\|_2^2}\left[G_{\ln}(v_\lambda)(x) - \frac{4}{N}\left(\int_{\mathbb{R}^N} G_{\ln}(v_\lambda) v_\lambda^2 ~dx\right)\right] v_\lambda.
    \end{split}
\end{equation*}
Hence, the proof $(ii)$ follows. Next, we show $(iii)$. 
Setting $z= y^{\ast,r}$, the Jacobian of this inversion satisfies
\begin{equation}\label{est:change-of-variables}
    dy = \left(\frac{r}{|z-\tilde{x}|}\right)^{2N} dz, \quad \quad \frac{r}{|y-\tilde{x}|} = \frac{|z-\tilde{x}|}{r}.
\end{equation}
This further gives
\begin{equation}\label{est-l2-norm}
   \begin{split}
    \|u^{\#}_r\|_2^2 &= \int_{\mathbb{R}^N} \left(\frac{r}{|y-\tilde{x}|}\right)^{2N} u^2\left(\frac{r^2 (y-\tilde{x})}{|y-\tilde{x}|^2} + \tilde{x}\right) ~dy\\
    & = \int_{\mathbb{R}^N} \left(\frac{|z-\tilde{x}|}{r}\right)^{2N} u^2(z) \left(\frac{r}{|z-\tilde{x}|}\right)^{2N} ~dz = \|u\|_2^2, \quad \text{for every} \ r>0.
\end{split}
\end{equation}
Using \eqref{parameterized:prob} and \cite[(2.4)]{Chen-Zhou_2025}, for $x \in \mathbb{R}^N \setminus \{\tilde{x}\}$, we have
\begin{equation}\label{est:conv-term-0}
\begin{split}
    & \mathcal{L}_\Delta u^{\#}_r (x) = \left(\frac{r}{|x-\tilde{x}|}\right)^N (\mathcal{L}_\Delta u)(x^{\ast,r}) - 4 \ln \left(\frac{|x-\tilde{x}|}{r}\right) u^{\#}_r (x) \\
    & \qquad= \left(\sigma - 4 \ln \left(\frac{|x-\tilde{x}|}{r}\right)\right) u^{\#}_r (x)  +  \frac{1}{\|u\|_2^2} \left[G_{\ln}(u)(x^{\ast,r}) - \frac{1}{2 \|u\|_2^2}\left(\int_{\mathbb{R}^N} G_{\ln}(u) u^2 ~dx\right) \right] u^{\#}_r (x).
    \end{split}
\end{equation}
Next, we derive the estimate on the last term in \eqref{est:conv-term-0} as follows: 
By change of variables $z= y^{\ast, r}$, using \eqref{est:change-of-variables} and \ref{est:iden-1}, we obtain
\begin{equation}\label{est:conv-term-1}
    \begin{split}
        & G_{\ln}(u^{\#}_r)(x) = \int_{\mathbb{R}^N} \ln\left(\frac{1}{|x-y|^4}\right) (u^{\#}_r)^2(y) ~dy = \int_{\mathbb{R}^N} \left(\frac{r}{|y-\tilde{x}|}\right)^{2N} \ln\left(\frac{1}{|x-y|^4}\right) u^2\left(y^{\ast, r}\right) ~dy\\
        &= \int_{\mathbb{R}^N} \left(\frac{r}{|y-\tilde{x}|}\right)^{2N} \left[\ln\left(\frac{1}{|x-\tilde{x}|^4}\right) - \ln\left(\frac{1}{|y^{\ast, r}-\tilde{x}|^4}\right) + \ln\left(\frac{1}{|x^{\ast, r}-y^{\ast, r}|^4}\right) \right] u^2\left(y^{\ast, r}\right) ~dy\\
        & = \int_{\mathbb{R}^N} \ln \left( \frac{1}{|x^{\ast,r} -z|^4}\right) u^2(z) ~dz - 4 \ln |x-\tilde{x}| \|u^{\#}_r\|_2^2 + 4 \int_{\mathbb{R}^N} u^2(z) \ln |z - \tilde{x}|  ~dz\\
        & =  G_{\ln}(u)(x^{\ast,r}) 
- 4 \ln \left(\frac{|x-\tilde{x}|}{r}\right) \|u^{\#}_r\|_2^2 + 4\int_{\mathbb{R}^N} u^2(z) \ln|z - \tilde{x}|\,dz  - 4 \ln r \|u^{\#}_r\|_2^2.
    \end{split}
\end{equation}
Now, by applying change of variables via \eqref{est:change-of-variables}, we obtain
\begin{equation}\label{est:conv-term-new}
\begin{split}
    \int_{\mathbb{R}^N} (u^{\#}_r)^2 \ln |u^{\#}_r| ~dx & = \int_{\mathbb{R}^N} \left(\frac{r}{|x-\tilde{x}|}\right)^{2N} u^2(x^{\ast,r}) \ln \left|\left(\frac{r}{|x-\tilde{x}|}\right)^N u(x^{\ast,r})\right| ~dx\\
    & = \int_{\mathbb{R}^N}  u^2(z) \ln | u(z)| ~dz + N \int_{\mathbb{R}^N} u^2(z) \ln\left(\frac{|z-\tilde{x}|}{r}\right) ~dz.
\end{split}
\end{equation}
By collecting the estimates in \eqref{est:conv-term-1} and \eqref{est:conv-term-new} and using \eqref{est-l2-norm}, we obtain
\begin{equation}\label{est:conv-term-3}
    \begin{split}
        & \frac{1}{\|u\|_2^2} \left[G_{\ln}(u)(x^{\ast,r}) - \frac{4}{N}\left(\int_{\mathbb{R}^N} u^2 \ln |u| ~dx\right) \right] \\
        & = \frac{1}{\|u^{\#}_r\|_2^2} \left[G_{\ln}(u^{\#}_r)(x) - \frac{4}{N}\int_{\mathbb{R}^N} (u^{\#}_r)^2 \ln |u^{\#}_r| ~dx\right]\\
        & \quad + \frac{1}{\|u^{\#}_r\|_2^2} \left[4 \ln \left(\frac{|x-\tilde{x}|}{r}\right) \|u^{\#}_r\|_2^2 - 4\int_{\mathbb{R}^N} u^2(z) \ln|z - \tilde{x}|\,dz  + 4 \ln r \|u^{\#}_r\|_2^2 + 4 \int_{\mathbb{R}^N} u^2(z) \ln\left(\frac{|z-\tilde{x}|}{r}\right) ~dz\right] \\
        & =  \frac{1}{\|u^{\#}_r\|_2^2}\left[ G_{\ln}(u^{\#}_r) - \frac{4}{N} \int_{\mathbb{R}^N} (u^{\#}_r)^2 \ln |u^{\#}_r| ~dx  \right] + 4 \ln \left(\frac{|x-\tilde{x}|}{r}\right) .
    \end{split}
\end{equation}
Finally, by inserting the estimate \eqref{est:conv-term-3} in \eqref{est:conv-term-0}, we obtain
\[
\mathcal{L}_\Delta u^{\#}_r (x) = \sigma u^{\#}_r (x) + \frac{1}{\|u^{\#}_r\|_2^2}\left[ G_{\ln}(u^{\#}_r) - \frac{4}{N} \int_{\mathbb{R}^N} (u^{\#}_r)^2 \ln |u^{\#}_r| ~dx \right] u^{\#}_r.
\]
Next, we show $(iv).$ Let $Q\in O(N)$ and define $v(x):=u(Qx)$. Since $Q$ is orthogonal, $|\det Q|=1$ and $|Qx|=|x|$ for every $x\in\mathbb{R}^N$. By the change of variables $z=Qx$, we obtain $\|v\|_2^2 = \|u\|_2^2.$ Now, by taking $z=Qy$, we have $dy=dz$ and, since $Q$ is orthogonal, $|x-Q^{-1}z| = |Qx-z|.$ Consequently,
\[
G_{\ln}(v)(x) = \int_{\mathbb{R}^N}
\ln\left(\frac{1}{|x-Q^{-1}z|^4}\right)u^2(z)\,dz =
\int_{\mathbb{R}^N}
\ln\left(\frac{1}{|Qx-z|^4}\right)u^2(z)\,dz =
G_{\ln}(u)(Qx).
\]
This further implies 
\[
\int_{\mathbb{R}^N} v^2 \ln v\,dx =
\int_{\mathbb{R}^N}
u^2(Qx) \ln u(Qx)\,dx = \int_{\mathbb{R}^N}
u^2 \ln u \,dz.
\]
Moreover, noting that for any $\xi \in \mathbb{R}^N$, we have
\[
\widehat{\mathcal{L}_\Delta v}(\xi) = 2\ln|\xi|\,\widehat{u}(Q\xi) =
2\ln|Q\xi|\,\widehat{u}(Q\xi) =
\widehat{\mathcal{L}_\Delta u}(Q\xi).
\]
Hence, $\mathcal{L}_\Delta v(x)
= (\mathcal{L}_\Delta u)(Qx)$ and the required claim follows.
\end{proof}





\subsection{Moving plane method}
For $\lambda \in \mathbb{R}$, the reflection of the point $x \in \mathbb{R}^N$ about $T_\lambda:=\{x \in \mathbb{R}^N : x_1 = \lambda\}$ is defined as
\[
x^\lambda = (2\lambda -x_1, x_2, \cdots, x_n) 
\]
and 
\[
 H_\lambda:= \{x \in \mathbb{R}^N : x_1 < \lambda\}, \qquad  H^\lambda:= \{x \in \mathbb{R}^N : x^\lambda \in H_\lambda\}.
\]
\noindent 
\textbf{Proof of Theorem \ref{thm-radi-mono-prop}:}
Let $u$ be the positive solution of \eqref{main:prob} and $u^{\#}_r$ be the Kelvin transform of the $u$ centered at $\tilde{x}$ and $r>0.$ Then, by Proposition \ref{pro:invariance}, $u^{\#}_r$ satisfies 
\[
\begin{cases}
\mathcal{L}_\Delta u^{\#}_r = \sigma u^{\#}_r  + \frac{1}{\|u^{\#}_r\|_2^2}\left[ G_{\ln}(u^{\#}_r) - \frac{4}{N}\left(\int_{\mathbb{R}^N} (u^{\#}_r)^2 \ln u^{\#}_r ~dx\right) \right] u^{\#}_r \quad & \text{in} \ \mathbb{R}^N \setminus \{\tilde{x}\},\\
u^{\#}_r \ge 0 & \text{in }\mathbb{R}^N \setminus \{\tilde{x}\}.
\end{cases}
\]
From Proposition \ref{pro:invariance} $(i)$-$(ii)$, without loss of generality, we can take $\tilde{x} =0$ and $r=1$ and denote $u_1:= u^{\#}_1.$ For $\lambda<0$ and $x \in \mathbb{R}^N \setminus \{0\}$, denote
\[
u_{1, \lambda}(x) = u_1(x^\lambda) \quad \text{and} \quad w_\lambda(x) = u_{1,\lambda}(x) - u_1(x).
\]
Now, for any $x \in H_\lambda$, using the reflection identities
\[ (y^\lambda)^\lambda =y \quad \text{and} \quad |x-z^\lambda| = |x^\lambda-z|, \]
and by change of variables $z= y^\lambda$, we obtain
\[
\begin{split}
    G_{\ln} (u_{1, \lambda}) (x) &= \int_{\mathbb{R}^N} \ln \left(\frac{1}{|x-y|^4} \right) (u_{1, \lambda}(y))^2 ~dy = \int_{\mathbb{R}^N} \ln \left(\frac{1}{|x-y|^4} \right) (u_{1}(y^\lambda))^2 ~dy\\ 
    & = \int_{\mathbb{R}^N} \ln \left(\frac{1}{|x-z^\lambda|^4} \right) (u_{1}(z))^2 ~dz = \int_{\mathbb{R}^N} \ln \left(\frac{1}{|x^\lambda-z|^4} \right) (u_{1}(z))^2 ~dz = G_{\ln} (u_{1}) (x^\lambda).
\end{split}
\]
This further implies
\[
\|u_{1, \lambda}\|_2 = \|u_1\|_2 \quad \text{and} \quad \int_{\mathbb{R}^N} G_{\ln}(u_1) u_1^2 ~dx = \int_{\mathbb{R}^N} G_{\ln} (u_{1, \lambda}) u_{1,\lambda}^2 ~dx. 
\]
Therefore, the function $w_\lambda$ satisfies
\begin{equation}\label{main-Kelvin-differ-0}
    \begin{cases}
\mathcal{L}_\Delta w_\lambda(x) + g_{\ln} w_\lambda(x) = \sigma  w_\lambda(x)  + \frac{1}{\|u_1\|_2^2} \left(G_{\ln} (u_{1, \lambda}) (x) u_{1, \lambda}(x) - G_{\ln} (u_{1}) (x) u_{1}(x) \right)  \quad & \text{in} \ H_\lambda,\\
w_\lambda(x_\lambda) = - w_\lambda(x) & \text{in } \ H_\lambda,
\end{cases}
\end{equation}
where 
\[
g_{\ln} := \frac{4}{N}\left(\int_{\mathbb{R}^N} (u^{\#}_r)^2 \ln u^{\#}_r ~dx\right).
\]
Next, we rewrite the difference of logarithmic potentials term in \eqref{main-Kelvin-differ-0} as follows:
\begin{equation}\label{est:form-1}
    \begin{split}
G_{\ln} (u_{1, \lambda}) (x) & u_{1, \lambda}(x) - G_{\ln} (u_{1}) (x) u_{1}(x) = \left(G_{\ln} (u_{1, \lambda}) (x) -G_{\ln} (u_{1}) (x)\right) u_{1, \lambda}(x) + G_{\ln} (u_{1}) (x) \left(u_{1, \lambda}(x) - u_{1}(x) \right)  \\
& = \left(\int_{\mathbb{R}^N} \ln \left(\frac{|x-z|^4}{|x^\lambda-z|^4} \right) (u_{1}(z))^2 ~dz \right) u_{1, \lambda}(x) + G_{\ln} (u_{1}) (x) \left(u_{1, \lambda}(x) - u_{1}(x) \right) \\
& = I_1(x) u_{1, \lambda}(x)  + I_2(x) u_{1, \lambda}(x) + G_{\ln} (u_{1}) (x) w_\lambda(x)
\end{split}
\end{equation}
where 
\[
I_1(x):= \int_{H_\lambda} \ln \left(\frac{|x-z|^4}{|x^\lambda-z|^4} \right) (u_{1}(z))^2 ~dz \quad \text{and}  \quad I_2(x):= \int_{H^\lambda} \ln \left(\frac{|x-z|^4}{|x^\lambda-z|^4} \right) (u_{1}(z))^2 ~dz.
\]
By change of variables and reflection identites, we have
\begin{equation}\label{est:I2}
   I_2(x) = \int_{H_\lambda} \ln \left(\frac{|x-z^\lambda|^4}{|x^\lambda-z^\lambda|^4} \right) (u_{1}(z^\lambda))^2 ~dz = - \int_{H_\lambda} \ln \left(\frac{|x-z|^4}{|x^\lambda-z|^4} \right) (u_{1,\lambda}(z))^2 ~dz.
\end{equation}
Combining \eqref{est:form-1} and \eqref{est:I2}, we obtain
\begin{equation}\label{differ-rewrite}
     \begin{split}
G_{\ln} (u_{1, \lambda}) (x) & u_{1, \lambda}(x) - G_{\ln} (u_{1}) (x) u_{1}(x) = I_\lambda(x) u_{1, \lambda}(x)  + G_{\ln} (u_{1}) (x) w_\lambda(x)
\end{split}
\end{equation}
where 
\[
I_\lambda(x):= \int_{H_\lambda} K_{\lambda,x,z} (u_{1}(z) + u_{1,\lambda}(z)) w_\lambda(z) ~dz, \quad K_{\lambda,x,z}:= \ln \left(\frac{|x-z^\lambda|^4}{|x-z|^4} \right).
\]
Note that the kernel in the above definition satisfies $K(\lambda,x,z) \geq 0$ for all $x,z \in H_\lambda.$
Then, $w_\lambda$ satisfies the equation
\begin{equation}\label{main-Kelvin-differ}
    \begin{cases}
\mathcal{L}_\Delta w_\lambda(x) + g_{\ln} w_\lambda(x) = \left(\sigma  + \frac{1}{\|u_1\|_2^2} G_{\ln} (u_{1}) (x)\right) w_\lambda(x)  + \frac{1}{\|u_1\|_2^2} I_\lambda(x) u_{1,\lambda}(x)  \quad & \text{in} \ H_\lambda,\\
w_\lambda(x_\lambda) = - w_\lambda(x) & \text{in } \ H_\lambda.
\end{cases}
\end{equation}
From \eqref{def:Kelvin-trans}, we observe that $w_\lambda$ vanishes at infinity and on the boundary of $H_\lambda.$ Moreover, if $ \inf_{x \in H_\lambda} w_\lambda(x)<0$, then there exists a point $\tilde{x}_{\lambda} \in H_\lambda$ such that
\[
w_\lambda(\tilde{x}_{\lambda}) = \inf_{x \in H_\lambda} w_\lambda(x) <0.
\]
Denote $$\Omega_\lambda:= \{x \in H_\lambda: w_\lambda(x) <0\}$$ and 
\[
w_{\lambda,-}(x):= 
\begin{cases}
\min\{0, w_\lambda(x)\} \ & \text{for} \ x \in H_\lambda,\\
0\ & \text{for} \ x \in \mathbb{R}^N \setminus H_\lambda.
\end{cases}
\]
Next, we start moving the plane $T_\lambda$ along the $x_1$-axis from near $-\infty$ to the right. \vspace{0.1cm}\\
\textbf{Claim 1:} There exists $\lambda^\ast \ll 1$ such that $w_\lambda >0$ in $H_\lambda$ for all $\lambda < \lambda^\ast.$
\vspace{0.1cm}\\
By multiplying the equation \eqref{main-Kelvin-differ} by $w_{\lambda,-}$ and integrating over $\mathbb{R}^N$, we obtain
\begin{equation}\label{test-func-est}
    \begin{split}
    I:= \int_{\mathbb{R}^N} & (\mathcal{L}_\Delta w_\lambda) w_{\lambda,-} ~dx = \int_{\mathbb{R}^N} \left( \sigma - g_{\ln} + \frac{1}{\|u_1\|_2^2} G_{\ln} (u_{1}) (x)\right) (w_{\lambda,-}(x))^2  ~dx \\
    & \qquad + \frac{1}{\|u_1\|_2^2} \int_{\mathbb{R}^N}  I_\lambda(x) u_{1,\lambda}(x)  w_{\lambda,-}(x) ~dx := K_1 + K_2. 
\end{split}
\end{equation}
Next, we derive the estimates for the above integrals seperately. First we show that for $x \in \supp(w_{\lambda,-})$
\begin{equation*}
    L_{\Delta} w_{\lambda, -}(x) \geq L_{\Delta} w_{\lambda}(x).
\end{equation*}
Note that
\[
\begin{split}
\frac{1}{c_N} \left(L_{\Delta} w_{\lambda, -}(x) - L_{\Delta} w_{\lambda}(x)\right) & = \int_{\mathbb{R}^N} \frac{w_\lambda(z) - w_{\lambda, -}(z)}{|x-z|^N} ~dz \\
& = \int_{H_\lambda \cap (\supp(w_{\lambda, -}))^c} \frac{w_\lambda(z)}{|x-z|^N} ~dz  + \int_{H_\lambda^c\cap (\supp(w_{\lambda, -}))^c} \frac{w_\lambda(z)}{|x-z|^N} ~dz \\
& = \int_{H_\lambda \cap (\supp(w_{\lambda, -}))^c} w_\lambda(z) \left[\frac{1}{|x-z|^N} - \frac{1}{|x-z^\lambda|^N}\right] ~dz \geq 0
\end{split}
\]
since $w_\lambda(z^\lambda) = - w_\lambda(z)$ and $|x-z^\lambda| \geq |x-z|$ for $x,z \in H_\lambda.$ This together with \cite[Theorem 1]{Beckner-1995} (by extending it to $\mathbb{H}_{\ln}(\mathbb{R}^N)$ functions via density arguments), we obtain
\begin{equation}\label{left-term-est}
   \begin{split}
    \int_{\mathbb{R}^N} & (\mathcal{L}_\Delta w_\lambda) w_{\lambda,-} ~dx \geq  \int_{\mathbb{R}^N}(\mathcal{L}_\Delta w_{\lambda, -}) w_{\lambda,-} ~dx = \int_{\mathbb{R}^N} \ln |\xi|^2 |\widehat{w}_{\lambda,-}|^2 ~d\xi \\
        & \geq 2 D_N \|w_{\lambda,-}\|_2^2 - 2 \int_{\mathbb{R}^N} \ln |x| |w_{\lambda,-}|^2 ~dx
\end{split} 
\end{equation}
where $D_N:= \psi\left(\frac{N}{4}\right) -\ln \pi.$
In order to derive estimates for $K_1$ and $K_2$, we show that there exist constants $C_\ast, C_\lambda >0$ and $ \lambda_\ast \ll -1$ such that
\begin{equation}\label{upper-est-I1}
    G_{\ln} (u_1) (x) \leq \ln |x|^{-4} \|u_1\|_2^2 + C_\ast \left(\|u_1\|_2^2 + O(|x|^{-N})\right) \ \text{for all} \ x \in H_\lambda \ \text{and} \ \lambda < \lambda_\ast
    \end{equation}
and    
\begin{equation}\label{est-norm-est-I2}
        \|I_\lambda^- u_{1,\lambda}\|_{2, H_\lambda} \leq C_\lambda \|w_{\lambda,-}\|_{2, H_\lambda} \quad \text{such that} \quad C_\lambda \to 0 \ \text{as} \ \lambda \to \infty
    \end{equation}
    where 
    \[
    I_\lambda^-(x) := \int_{H_\lambda} K(\lambda, x, z)(u_{1}(z) + u_{1,\lambda}(z)) w_{\lambda,-}(z) ~dz.
    \]
Let $x \in H_\lambda.$ By applying change of variables, we obtain
\[
\begin{split}
G_{\ln}(u_1)(x) & = \ln |x|^{-4} \|u_1\|_2^2 - 4 \int_{\mathbb{R}^N} \ln\left(\frac{x}{|x|} - \frac{y}{|x|}\right) u_1^2(y) ~dy \\
& = \ln |x|^{-4} \|u_1\|_2^2 - 4 \int_{|y| < \frac{|x|}{2}} \ln\left(\frac{x}{|x|} - \frac{y}{|x|}\right) u_1^2(y) ~dy -4  \int_{|y| \geq \frac{|x|}{2}} \ln\left(\frac{x}{|x|} - \frac{y}{|x|}\right) u_1^2(y) ~dy\\
& := \ln |x|^{-4} \|u_1\|_2^2 -4 (J_1 + J_2).
\end{split}
\]
It is easy to that for $|y| < \frac{|x|}{2}$, $\frac{1}{2} \leq \left(\frac{x}{|x|} - \frac{y}{|x|}\right) \leq  \frac{3}{2}$, which further implies $J_1 \leq C \|u_1\|_2^2$. Now, by the definition of Kelvin transform, we have
\begin{equation}\label{Kelvin-prop}
    \lim_{|x| \to \infty} u_1 (x) |x|^N = c \quad \text{for some} \ c \in \mathbb{R}^+.
\end{equation}
For any $\lambda \ll -1$ and $x\in H_\lambda$, using \eqref{Kelvin-prop}, we get $u_1(y)\leq 2c|y|^{-N}$ in $\{|y|\geq \frac{|x|}{2}\}$ and applying the change of variable $z = \frac{y}{|x|}$ and using the notation $e = \frac{x}{|x|}$, we obtain
\begin{equation}\label{est-J-2}
  J_2 \leq 2c \int_{|y| \geq \frac{|x|}{2}} \left|\ln\left(\frac{x}{|x|} - \frac{y}{|x|}\right)\right| |y|^{-2N} ~dy = C_0 |x|^{-N}  \end{equation}
where 
$
C_0:= 2c \int_{|z| \geq \frac{1}{2}} \left|\ln\left(e - z\right) \right| |z|^{-2N} ~dz < +\infty.$
Hence, the required claim in \eqref{upper-est-I1}. By the triangle inequality and the definition of the reflection $z^\lambda$, we have
\[
\frac{|x-z^\lambda|}{|x-z|}
\leq 1+\frac{|z-z^\lambda|}{|x-z|}
=1+\frac{2(\lambda-z_1)}{|x-z|}\quad \text{for all $x,z\in H_\lambda$ with $x\neq z$.}
\]
Consequently, for every $\alpha \in (0,\frac{1}{2})$, there exists $C_\alpha >0$ such that 
\begin{equation}\label{kernel-bound}
    K(\lambda, x, z) \leq 4 \ln \left(1+ \frac{2(\lambda-z_1)}{|x-z|}\right) \leq C_\alpha \frac{(\lambda-z_1)^\alpha}{|x-z|^{ \alpha}} \quad \text{for all} \ \ x,z \in H_\lambda \ \text{and} \ x \neq z.
\end{equation}
By using \eqref{kernel-bound} and applying Cauchy-Schwarz inequality, we obtain
\begin{equation}\label{esti-I-}
    \begin{split}
    \|I_\lambda^- u_{1,\lambda} \|^2_{2,H_\lambda} & = \int_{H_\lambda} (u_{1,\lambda}(x))^2 \left(\int_{H_\lambda} K(\lambda, x, z)(u_{1}(z) + u_{1,\lambda}(z)) w_{\lambda,-}(z) ~dz\right)^2 ~dx \\
    & \leq \|w_{\lambda, -}\|_{2, H_\lambda}^2 \int_{H_\lambda} (u_{1,\lambda}(x))^2 \int_{\supp(w_{\lambda,-})} (K(\lambda,x,z))^2 (u_1(z) + u_{1,\lambda}(z))^2 ~dz ~dx\\
    & \leq C_\lambda^2 \|w_{\lambda, -}\|_{2, H_\lambda}^2
\end{split}
\end{equation}
where 
\begin{equation}\label{def:C-lambda}
    C_\lambda^2:= C_\alpha \int_{\supp(w_{\lambda,-})} (\lambda-z_1)^{2\alpha} (u_1(z) + u_{1,\lambda}(z))^2 \int_{H_\lambda} \frac{(u_{1,\lambda}(x))^2}{|x-z|^{2\alpha}}  ~dx ~dz.
\end{equation}
It is easy to derive that 
\begin{equation}\label{part-est-1}
   \begin{split}
    \int_{H_\lambda} \frac{(u_{1,\lambda}(x))^2}{|x-z|^{2\alpha}}  ~dx \leq C_1 \|u^2_{1,\lambda}\|_{L^\infty(H_\lambda)}  + \|u_{1,\lambda}\|_{2,H_\lambda}^2\quad \text{where} \ C_1:= \int_{H_\lambda \cap B_1(z)} \frac{1}{|x-z|^{2\alpha}}  ~dx.
\end{split} 
\end{equation}
Now, by taking $\lambda \ll -1$ and using \eqref{part-est-1} in \eqref{def:C-lambda}, \eqref{Kelvin-prop}, $u_{1,\lambda} < u_1$ in $\supp(w_{\lambda,-})$ and $|\lambda-z_1| \leq |z|$ for $z \in H_\lambda$, we obtain
\begin{equation*}
    \begin{split}
        C_\lambda^2 & \leq C_1 \left(\|u^2_{1,\lambda}\|_{L^\infty(H_\lambda)} + \|u_{1,\lambda}\|_{2,H_\lambda}^2\right) \int_{\supp(w_{\lambda,-})} (\lambda-z_1)^{2\alpha} (u_1(z))^2 ~dz\\
        & \leq C_2 \left(\|u^2_{1,\lambda}\|_{L^\infty(H_\lambda)} + \|u_{1,\lambda}\|_{2,H_\lambda}^2\right) \int_{\supp(w_{\lambda,-})} |z|^{2\alpha - 2N} ~dz\\
        & \leq C_2 \left(\|u^2_{1,\lambda}\|_{L^\infty(H_\lambda)} + \|u_{1,\lambda}\|_{2,H_\lambda}^2\right) \int_{|z| > |\lambda|} |z|^{2\alpha - 2N} ~dz\\
        & \leq C_3 \left(\|u^2_{1,\lambda}\|_{L^\infty(H_\lambda)} + \|u_{1,\lambda}\|_{2,H_\lambda}^2\right) |\lambda|^{2\alpha-N} \to 0 \ \text{as} \ \lambda \to - \infty
    \end{split}
\end{equation*}
where the second last inequality follows from the fact that, for $z\in H_\lambda$ and $\lambda<0$, we have
\[
|\lambda|<|z_1|\leq |z|,
\]
and consequently,
\[
\supp(w_{\lambda,-})\subset H_\lambda
\subset \{z\in\mathbb{R}^N: |z|>|\lambda|\}.
\]
Hence, the required claim in \eqref{est-norm-est-I2}. Now, by using \eqref{upper-est-I1} and \eqref{est-norm-est-I2} in \eqref{test-func-est}, we obtain
\begin{equation}\label{right-term-est}
    \begin{split}
        K_1 + K_2 \leq (\sigma - g_{\ln}) \|w_{\lambda,-}\|_2^2 -  4 \int_{\mathbb{R}^N} \ln |x| |w_{\lambda,-}|^2 ~dx + (C_\ast + \frac{C_\lambda}{\|u_1\|_2^2}) \|w_{\lambda,-}\|_2^2.
    \end{split}
\end{equation}
Finally, combining the estimates in \eqref{left-term-est} and \eqref{right-term-est}, we obtain for $\lambda \ll -1$ and $C^\ast >0$ (independent of $\lambda$)
\begin{equation}\label{contr:ineq}
    (2 D_N - \sigma + g_{\ln} - C^\ast) \|w_{\lambda,-}\|_2^2  \leq  - 2 \int_{\mathbb{R}^N} \ln |x| |w_{\lambda,-}|^2 ~dx
\end{equation}
which leads to $w_{\lambda,-} \equiv 0$ in $H_\lambda$ for $\lambda \ll -1.$ Therefore, $w_\lambda \geq 0$ in $H_\lambda$. This further yields $ I_\lambda(x) \geq 0$ in $H_\lambda.$ Hence, by applying the strong maximum principle \cite[Lemma~4.1]{Chen-Zhou_2025} we obtain
\[
w_\lambda>0 \quad \text{or} \quad w_\lambda \equiv 0 \qquad \text{in } H_\lambda.
\]
Next, we define
\begin{equation}\label{def-Lambda}
    \Lambda
:=
\sup\left\{
\lambda<0:\;
w_\mu(x)\ge0,\ \forall\,x\in H_\mu,\ \forall\,\mu\le\lambda
\right\}
\end{equation}
and show that $u_1$ is symmetric with respect to the plane $T_\Lambda$. \\
\textbf{Case 1: $\Lambda =0$}. Then, by the continuity of $u_1$, we obtain
\[
u_1(x_1,x') \leq u_1(-x_1,x'), \qquad x_1<0,\; x'\in\mathbb{R}^{N-1}.
\]
For $\lambda \geq 0$, we move the plane from $x_1=+\infty$ toward the left. By an argument analogous to the one used above, we obtain
\[
u_1(-x_1,x') \geq u_1(x_1,x'),
\qquad x_1>0,\; x'\in\mathbb{R}^{N-1}.
\]
Combining the two inequalities yields, we obtain 
\[
u_1(x)=u_{1}(x^0),
\qquad x\in H_{0}.
\]
Moreover, since the moving plane argument can be carried out in every direction until the limiting position passes through the origin, we conclude that $u_1$ is a radial function about its center $\tilde{x} =0$. 

Next, suppose first that $u_1$ is bounded at $\tilde{x} =0$, say $L_{\tilde{x}}:=u_1(\tilde{x})$. Then the radial symmetry and monotonicity imply that
\[
L_{\tilde{x}}
=
\lim_{x \to \tilde{x}} u_1(x)
\in(0,+\infty) \ \text{and} \ \lim_{|z|\to\infty} u(z)|z-\tilde{x}|^N = \lim_{|z|\to\infty} u(z)|z|^N
= L_{\tilde{x}}.
\]
Moreover, if we choose another center $\overline{x}$, then the above computation gives
\[
\lim_{|z|\to\infty} u(z)|z-\overline{x}|^N = \lim_{|z|\to\infty} u(z)|z|^N
= L_{\overline{x}}.
\]
Therefore, $L_{\overline{x}} = L_{\tilde{x}}.$
This further gives \eqref{asym-beha}. Next, suppose that $u_1$ is unbounded at $\tilde{x} =0$. By radial symmetry, $u$ must blow up at $\tilde{x}$, and therefore
\[
\lim_{|x|\to\infty} u(x)|x|^N
=
+\infty
\]
which further contradicts the assumption that $u \in L^1_0(\mathbb{R}^N).$ Hence, the claim in \eqref{asym-beha} follows.\\
\textbf{Case 2: $\Lambda < 0$.}
By the continuity of $w_\lambda$ with respect to $\lambda$ and \cite[Lemma 4.1]{Chen-Zhou_2025}, one of the following alternatives holds:
\[
w_{\Lambda}\equiv0
\quad\text{in }H_{\Lambda} \quad \text{or} \quad
w_{\Lambda}>0
\quad\text{in }H_{\Lambda}.
\]
If $w_{\Lambda}\equiv 0$ in $H_{\Lambda}$ then
\[
u_1(x)=u_{1,\Lambda}(x),
\qquad x\in H_{\Lambda} \quad \text{and} \ x^\Lambda \not \equiv 0.
\]
Consider a sequence $x_k \to 0$ such that $x_k \neq 0$ and $x_k^{\Lambda} \in H_{\Lambda}.$
Now, by the symmetry of $u_1$ with respect to $T_{\Lambda}$, we have
\[
u_1(x_k)=u_1(x_k^{\Lambda}).
\]
As $x_k\to 0$, we have
\[
x_k^{\Lambda}\to 0^{\Lambda}
=(2\Lambda,0,\ldots,0).
\]
Since $\Lambda<0$, it follows that $0^{\Lambda}\neq 0.$ Therefore, $0^{\Lambda}$ is a regular point of $u_1$, and by the
continuity of $u_1$ in $\mathbb{R}^N\setminus\{0\}$, we obtain
\[
u_1(x_k) = u_1(x_k^{\Lambda})
\longrightarrow
u_1(0^{\Lambda})\in(0,\infty).
\]
Consequently,
\[
\lim_{x\to 0}u_1(x)
=
u_1(0^{\Lambda})
\in(0,\infty).
\]
Thus, $u_1$ admits a positive and finite continuous extension to the
Kelvin center $0$. In particular, after defining $u_1(0):=\lim\limits_{x \to 0}u_1(x),$ we have $u_1(0)=u_1(0^{\Lambda})\in(0,\infty).$ Now, by taking $\tilde{x} =0$ and $r=1$ and putting $y=\frac{x}{|x|^2}$, we obtain
\[
u_1(x)
=
|x|^{-N} u \!\left(\frac{x}{|x|^2}\right) = |y|^N u(y).
\]
Note that $|y|\to\infty$ as $x\to0$. Hence,
\[
\lim\limits_{x \to 0} u_1(x) = \lim_{|y| \to \infty} |y|^N u(y)<\infty,
\]
which yields \eqref{asym-beha} and 
\[
u_1(x)=u_{1,\Lambda}(x),
\qquad x\in H_{\Lambda}.
\]
Therefore, $u_1$ is symmetric with respect to the plane $T_\Lambda.$

Next, we show that the second alternative, namely, $w_{\Lambda}(x)>0$ for all $x\in H_{\Lambda}$, cannot occur when $\Lambda<0$. To this end, it suffices to prove that there exists $\varepsilon>0$, sufficiently small, such that
\[
w_\lambda(x)\ge0,
\qquad \text{for all}\ \,x\in H_\lambda,\quad
\lambda\in[\Lambda,\Lambda+\varepsilon),
\]
which contradicts the definition of $\Lambda$. Now, by using \eqref{Kelvin-prop} for $x \in H_\lambda \setminus B_R$ and $R \gg 1$, we get $u_1(y)\leq 2c|y|^{-N}$ in $\{|y|\geq \frac{|x|}{2} \geq \frac{R}{2}\}$ and the upper bounds in \eqref{est-J-2}. Therefore, by repeating the arguments in the proof of \eqref{upper-est-I1}, there exists $R \gg 1$ such that 
\begin{equation}\label{claim-2-est-1}
    G_{\ln} (u_1) (x) \leq \ln |x|^{-4} \|u_1\|_2^2 + C_\ast \left(\|u_1\|_2^2 + O(|x|^{-N})\right) \  \text{for all} \ x \in H_\lambda \setminus B_R \ \text{and} \ \lambda \in [\Lambda, \Lambda +\eps).
\end{equation}
From \eqref{esti-I-}, we obtain
\begin{equation}\label{claim-2-est-2}
\|I_\lambda^- u_{1,\lambda}\|_{2, H_\lambda} \leq C_\lambda \|w_{\lambda,-}\|_{2, H_\lambda} 
\end{equation}
where $C_\lambda$ is defined in \eqref{def:C-lambda}. Next, we show that $ C_\lambda \to 0$ as $ \lambda \to \Lambda.$
Now, by using the continuity of $w_\lambda$, there exists a sequence $\{R_{\lambda}\}_{\lambda \in (\Lambda, \Lambda +\eps)}$ such that $R<R_\lambda$, $ R_\lambda \to +\infty$ as $\lambda \to \Lambda$, and 
\[
w_{\lambda} > 0 \quad \text{for all} \ x \in H_\lambda \cap B_{\lambda}(0) \ \text{and} \ \lambda \in (\Lambda, \Lambda +\eps).
\]
Therefore, $\supp(w_{\lambda,-}) \subset H_\lambda \setminus B_{|\lambda|}(0)$ for every $\lambda \in (\Lambda, \Lambda +\eps)$. Moreover, $|\lambda-z_1| \leq |z|$ for $z \in \supp(w_{\lambda,-})$. Since $w_\lambda \to w_\Lambda$ as $\lambda \to \Lambda$ and $w_\Lambda >0$ in $H_\Lambda$, it follows that $\chi_{\supp(w_{\lambda,-})}(z) \to 0$ for a.e. $z \in \mathbb{R}^N$. Furthermore, by the continuity of $u_1$ and \eqref{Kelvin-prop}, the function $\chi_{\supp(w_{\lambda,-})} |z|^{2\alpha} (u_1(z))^2$ is dominated by an integrable function for any $\alpha \in (0, \frac{1}{2})$. Hence, using \eqref{def:C-lambda} and \eqref{part-est-1}, and applying the dominated convergence theorem, we obtain
\[
\begin{split}
    C_\lambda^2 & \leq C_1 \left(\|u^2_{1,\lambda}\|_{L^\infty(H_\lambda)} + \|u_{1,\lambda}\|_{2,H_\lambda}^2\right) \int_{\mathbb{R}^N} \chi_{\supp(w_{\lambda,-})} |z|^{2\alpha} (u_1(z))^2 ~dz \to 0 \ \text{as} \ \lambda \to \Lambda.
\end{split}
\]
Again, by multiplying the equation \eqref{main-Kelvin-differ} by $w_{\lambda,-}$ for $\lambda \in (\Lambda, \Lambda +\eps)$ and integrating over $\mathbb{R}^N$ and using \eqref{claim-2-est-1}-\eqref{claim-2-est-2}, repeating the same arguments as in \textbf{Claim 1}, we obtain $\|w_{\lambda,-}\|_{2,H_\lambda} = 0$ for $\lambda \in (\Lambda, \Lambda +\eps)$, which is a contradiction to the definition of $\Lambda.$ 

Now, fix the standard basis vectors $e_1,\ldots,e_N$ of $\mathbb R^N$. For each $i\in{1,\ldots,N}$, we apply the moving plane method in the $e_i$-direction. Let $T_\lambda^i:=\{x\in\mathbb R^N:x_i=\lambda\}$, and denote by $x^{\lambda,i}$ the reflection of $x$ with respect to $T_\lambda^i$. Set $w_\lambda^i(x):=u_1(x^{\lambda,i})-u_1(x)$.

In view of \textbf{Case 2}, let $\Lambda_i$ be the limiting position in the $e_i$-direction such that $w_{\Lambda_i}^i\equiv 0$ on $T^i_{\Lambda_i}$. Then $u_1$ is symmetric with respect to the hyperplane $T_{\Lambda_i}^i$. Moreover, from the strict inequality before reaching the limiting position, $w_\lambda^i(x)>0$ for $\lambda<\Lambda_i$, it follows that $u_1$ is strictly increasing in the $x_i$-direction for $x_i<\Lambda_i$. By symmetry with respect to $T_{\Lambda_i}^i$, $u_1$ is strictly decreasing in the $x_i$-direction for $x_i>\Lambda_i$.

Define $x_0:=(\Lambda_1,\ldots,\Lambda_N)$. It is easy to see that $x_0$ is the unique maximum point of $u$.


Next, we apply the moving-plane argument in an arbitrary direction. In view of Proposition \ref{pro:invariance} (iv), the equation is invariant under rotations, it is enough to reduce the general direction to the already treated coordinate direction. Let $e\in \mathbb S^{N-1}$ be arbitrary. Choose an orthogonal matrix $Q$ such that $Qe_1=e$, and define $v(x):=u_1(Qx)$. By the rotational invariance of the equation, $v$ satisfies the same equation as $u_1$. Hence, we may apply the moving-plane argument already established in the $e_1$-direction to $v$. Therefore, there exists $\Lambda_e \in\mathbb R$ such that $v$ is symmetric with respect to the hyperplane $\{x\in\mathbb R^N:x_1=\Lambda_e\}$. Equivalently, $v(x)=v(x^{\Lambda_e,1})$, where $x^{\Lambda_e,1}$ denotes the reflection of $x$ with respect to the hyperplane ${x_1=\Lambda_e}$.

Returning to the original variables $y=Qx$, we observe that $x_1=x\cdot e_1=(Qx)\cdot(Qe_1)=y\cdot e$. Thus, the hyperplane ${x_1=\Lambda_e}$ is transformed by $Q$ into $T_{\Lambda_e}^e:=\{y\in\mathbb R^N:y\cdot e=\Lambda_e\}$. Consequently, $u_1$ is symmetric with respect to $T_{\Lambda_e}^e$.

Now, let $x_0$ denote the unique maximum point of $u_1$ obtained from the moving-plane argument in the standard coordinate directions. Since $u_1$ is symmetric with respect to $T_{\Lambda_e}^e$, the reflection of $x_0$ across $T_{\Lambda_e}^e$ is also a maximum point of $u_1$. By uniqueness of the maximum, this reflected point must coincide with $x_0$. Therefore, $x_0\in T_{\Lambda_e}^e$, that is, $x_0\cdot e=\Lambda_e$.

Since $e\in\mathbb S^{N-1}$ was arbitrary, $u_1$ is symmetric with respect to every hyperplane $\{x\in\mathbb R^N:(x-x_0)\cdot e=0\}$ passing through $x_0$. It follows that $u_1$ is radially symmetric and radially decreasing with respect to the point $x_0=(\Lambda_1,\ldots,\Lambda_N)$.

Now, assuming the decay condition \eqref{asym-beha}, we apply the moving plane method directly to the solution $u$ of \eqref{main:prob}. This yields that $u$ is radially symmetric 
about some point and is radially decreasing. Finally, by testing the equation \eqref{main:prob} with $u$ and using Proposition \ref{new:Log-Choq:ineq}, we obtain
\begin{equation}\label{testing-eq}
 \begin{split}
        \int_{\mathbb{R}^N} \ln |\xi|^2 |\mathcal{F}(u)|^2 ~d\xi & = \sigma \|u\|_2^2 + \frac{1}{\|u\|_2^2 }\int_{\mathbb{R}^N}  \left(\ln \left(\frac{1}{|x|^4}\right) \ast u^2\right) u^2(x) ~dx ~dy - \frac{4}{N} \int_{\mathbb{R}^N} u^2 \ln |u| ~dx \\
        & \geq B_{N,\mathcal{L}} \|u\|_2^2 +  \frac{1}{\|u\|_2^2} \int_{\mathbb{R}^N} \left(\ln \left(\frac{1}{|x|^4}\right) \ast u^2\right) u^2(x) ~dx ~dy \\
        & \qquad - \frac{4}{N}  \int_{\mathbb{R}^N} u^2(x) \ln |u(x)|~dx + \frac{4}{N} \|u\|_2^2 \ln \|u\|_2.
    \end{split}
    \end{equation}
This implies
\[
(\sigma-B_{N,\mathcal{L}}) \|u\|_2^2 \geq \frac{4}{N} \|u\|_2^2 \ln \|u\|_2 \quad \Longleftrightarrow \quad \|u\|_2 \leq e^{\frac{N}{4}\left(\sigma- B_{N, \mathcal{L}}\right)}.
\]
\qed
\section{The uniqueness of solutions}\label{uniqueness}
First, we show that the solution $u$ of the problem \eqref{main:prob} satisfies the self-Kelvin-duality identity in the sense that: the solution $u$ is a fixed point of its own Kelvin transform about a sphere centered at some point $x_0 \in \mathbb{R}^N$ and of some radius (depending upon $u_\infty$ and $u(x_0)$).
\begin{lemma}\label{lem:self-dual-iden}
    Let $u$ be the positive solution of the problem \eqref{main:prob}, $x_0 \in \mathbb{R}^N$ and $s = \left( \frac{u_\infty}{u(x_0)}\right)^\frac{1}{N}.$ Then, 
    \begin{equation}\label{self-dual-identity}
            u(x) = \frac{s^N}{|x-x_0|^N} u\left(\frac{s^2 (x-x_0)}{|x-x_0|^2} + x_0\right) = u_s^{\#}(x) \quad \text{for all} \ x \in \mathbb{R}^N
    \end{equation}
where $u_\infty$ is defined in \eqref{asym-beha}.    
\end{lemma}
\begin{proof}
By change of variable, it is enough to show that
\begin{equation}\label{self-dual-identity-1}
    u(sx +x_0) = \frac{1}{|x|^N} u\left(x_0 + s \frac{x}{|x|^{2}}\right) \quad \text{for all} \ x \in \mathbb{R}^N.
\end{equation}
   Without loss of generality, let $x_0=0$ and $e$ be a unit vector in $\mathbb{R}^N$. Define 
   \[ 
   w(x)=|x|^{-N}u\left(\frac{x}{|x|^2}-se\right). 
   \] 
Applying Proposition \ref{pro:invariance} $(i)$ and $(iii)$, $w$ is also a solution of \eqref{main:prob} in $\mathbb{R}^N \setminus \{0\}$. Note that $0$ is the removable singularity, since $w(0) = u_\infty.$ Indeed, we have
   \[ 
   \begin{split} w(0) &=\lim_{|x|\to 0^+} w(x) =\lim_{|x|\to 0^+}|x|^{-N} u\left(\frac{x}{|x|^2}-s e\right) =\lim_{|z|\to +\infty}|z|^N u(z-s e) =u_\infty. 
   \end{split} 
   \]
Therefore, by Theorem \ref{thm-radi-mono-prop}, $w$ is radially symmetric and monotically decreasing about some point in $\mathbb{R}^N.$ Next, we show that the point is exactly $\frac{e}{2s}.$ Since, $u$ is radial about $0$, it is invariant under all rotations fixing the axis $\R e$ and these rotations commute with the inversion $x \to \frac{x}{|x|^2}$ and translation by $-se.$ Therefore, $w$ inherits axial symmetry about $\R e.$ Moreover, we have 
   \[ 
   \begin{split} w(0) =u_\infty=s^N u(0) =w\left(\frac{e}{s} \right). 
   \end{split} 
   \]
This further implies that the radial center of $w$ is $\frac{e}{2s}.$ For $h>0$, we define two points $z_1, z_2$ symmetric about $\frac{e}{2s}$ as following:
\[ z_1=\frac{e\left(\frac{1}{2}-h\right)}{s}, \qquad z_2=\frac{e\left(\frac{1}{2}+h\right)}{s} \quad \text{such that} \ z_1 + z_2 = \frac{e}{s}. \] Moreover, \[ 
    \frac{z_1}{|z_1|^2}-s e =s\frac{\left(\frac{1}{2}+h\right)}{\left(\frac{1}{2}-h\right)}e, \quad \text{and} \quad \frac{z_2}{|z_2|^2}-s e = s\frac{\left(\frac{1}{2}-h\right)}{\left(\frac{1}{2}+h\right)}e. \] Since $w$ is symmetric about $\frac{e}{2s}$, we have \[ s^N \left|\frac{1}{2}-h\right|^{-N} u\left(s\frac{\left(\frac{1}{2}+h\right)}{\left(\frac{1}{2}-h\right)}e\right) = w(z_1)=w(z_2) =  s^N\left|\frac{1}{2}+h\right|^{-N} u\left(s\frac{\left(\frac{1}{2}-h\right)}{\left(\frac{1}{2}+h\right)}e\right). \] 
Finally, by taking $t=\frac{\frac{1}{2}-h}{\frac{1}{2}+h}$, we obtain \[ u(st e)=t^{-n}u(st^{-1}e). \] Since $e$ is arbitrary, \eqref{self-dual-identity} holds for all $x \in \mathbb{R}^N.$ The general case $x_0\neq 0$ follows by translation. 
\end{proof}
Note that $u_{\sigma,t}$ defined in \eqref{extremals} satisfies $u_{\sigma,t}(x_0)(u_{\sigma,t})_\infty = e^{\frac{N}{2} (\sigma-B_{N, \mathcal{L}})}B_{N,0}^2 $ and by Lemma \ref{constants-value}, we have $\|u_{\sigma,t}\|_2 = e^{\frac{N}{4}(\sigma-B_{N,\mathcal{L}})}.$ Next, we show that every normalized solution $u$ of \eqref{main:prob} with $\|u\|_2 =e^{\frac{N}{4} (\sigma-B_{N, \mathcal{L}})}$ satisfies the same relation.
\begin{lemma}\label{end-point-fixing}
    Let $\sigma \in \mathbb{R}$ and $u$ be a positive solution of \eqref{main:prob} such that $u$ verifies $\|u\|_2 = e^{\frac{N}{4} (\sigma-B_{N, \mathcal{L}})}$ and is radial symmetric
about $x_0$, decreasing in the radial direction $|x-x_0|$, then
\[
u(x_0) u_\infty = e^{\frac{N}{2} (\sigma-B_{N, \mathcal{L}})}B_{N,0}^2.
\]
\end{lemma}
\begin{proof}
In view of Proposition \ref{pro:invariance} $(i)$ and $(ii)$, we may assume, without loss of generality, that $x_0=0$, $t=1$ and $u_\infty= e^{\frac{N}{4} (\sigma-B_{N, \mathcal{L}})}B_{N,0}.$ It remains to prove that $u(0)=  e^{\frac{N}{4} (\sigma-B_{N, \mathcal{L}})}B_{N,0}.$ We first consider the case $u(0) > e^{\frac{N}{4} (\sigma-B_{N, \mathcal{L}})}B_{N,0}.$ For a fixed unit vector $e \in \mathbb{R}^N$ and taking $t=1$ in \eqref{extremals}, define 
\[ \tilde{w}(\theta)=u(\theta e), \qquad \tilde{u}(\theta)=u_{\sigma,1}(\theta e), \qquad \theta\in\mathbb{R}. 
\] 
Since both $u$ and $u_{\sigma,1}$ are radially symmetric, the functions $\tilde{w}$ and $\tilde{u}$ are independent of the choice of $e$. By the continuity of the functions $u$ and $u_{\sigma,1}$, the normalization $\|u\|_2 = \|u_{\sigma,1}\|_2 = e^{\frac{N}{4} (\sigma-B_{N, \mathcal{L}})}$ and using the fact that $u_\infty= e^{\frac{N}{4} (\sigma-B_{N, \mathcal{L}})}B_{N,0} = (u_{\sigma,1})_\infty$, there exists $\theta_0>0$ such that \begin{equation}\label{conti-est-1} 
\tilde{w}(\theta)>\tilde{u}(\theta), \qquad |\theta|<\theta_0, \qquad \tilde{w}(\pm\theta_0)=\tilde{u}(\pm\theta_0).
\end{equation}  
Set $s^N=\frac{e^{\frac{N}{4} (\sigma-B_{N, \mathcal{L}})}B_{N,0}}{\tilde{u}(\theta_0)}$. Now, by using the definition of $u_{\sigma,1}$ and \eqref{self-dual-identity-1} with $x_0=\theta_0e$, gives $s = s_0:=\sqrt{1+\theta_0^2}$ and \begin{equation}\label{appli-lem-5.2} 
\tilde{w}(s_0 \theta+\theta_0) =|\theta|^{-N}       \tilde{w}(s_0\theta^{-1}+\theta_0), \qquad \tilde{u}(s_0\theta+\theta_0) =|\theta|^{-N} \tilde{u}(s_0\theta^{-1}+\theta_0). 
\end{equation}  
Observe that \[ -\theta_0<s_0\theta^{-1}+\theta_0<\theta_0 \Longrightarrow -\infty<s_0\theta+\theta_0 < - g(\theta_0):= - \theta_1, \quad \text{where} \ g(\theta):= \frac{1-\theta^2}{2\theta}. \] Consequently, by \eqref{conti-est-1}-\eqref{appli-lem-5.2} and the symmetry of $u$ and $u_{\sigma,t}$, we obtain 
\begin{equation}\label{extens-interval}
\tilde{w}(\theta)>\tilde{u}(\theta), \ \text{for} \ \theta \in (-\infty,- \theta_1) \cup (\theta_1, +\infty) \cup (-\theta_0, \theta_0), \quad \text{and} \quad \tilde{w}(\pm \theta_1)=\tilde{u}(\pm \theta_1), \quad \text{where} \ \theta_1:=g(\theta_0).
\end{equation} 
Note that $\theta_1 < \theta_0$ is equivalent to $ \theta_0 > \frac{1}{\sqrt{3}}.$ Therefore, if $\theta_0 > \frac{1}{\sqrt{3}}$, then 
\[
\tilde{w}(\theta)> \tilde{u}(\theta) \quad \text{for all} \ \theta \in \mathbb{R} \ \Longrightarrow \ u(x) > u_{\sigma,1}(x) \quad \text{for all} \ x \in \mathbb{R}^N
\]
which contradicts the fact that $\|u\|_2 = \|u_{\sigma,1}\|_2=.$
Next, if $\theta_0= \frac{1}{\sqrt{3}}$, then $\theta_1 = \theta_0$ and by \eqref{conti-est-1} together with \eqref{extens-interval} gives 
\[ \tilde{w}(\theta)>\tilde{u}(\theta), \qquad \theta\in(-\infty,\theta_0)\setminus\{\pm\theta_0\}, \quad \tilde{w}(\pm\theta_0)=\tilde{u}(\pm\theta_0). 
\] 
Equivalently, \[ u(x)>u_{\sigma,1}(x), \qquad x\in\mathbb{R}^N\setminus\{|x|=\theta_0\}, \ \text{and} \  u(x)=u_{\sigma,1}(x), \qquad |x|=\theta_0, \] 
which again contradicts the fact that $\|u\|_2 = \|u_{\sigma,1}\|_2=e^{\frac{N}{2} (\sigma-B_{N, \mathcal{L}})}.$ 
Now, the only case remaining is $\theta_0 <\frac{1}{\sqrt{3}}$. \\
\textbf{Case 1.} $\theta_0<\frac{1}{\sqrt{3}}$ and $\theta_1 \leq 1.$ Now, by repeating the above arguments with $s^N = \frac{e^{\frac{N}{4} (\sigma-B_{N, \mathcal{L}})}B_{N,0}}{\tilde{u}(\theta_1)}$ and using \eqref{self-dual-identity-1} with $x_0=\theta_1 e$, gives $s = s_1:=\sqrt{1+\theta_1^2}$ and \begin{equation}\label{appli-lem-5.2-1} 
\tilde{w}(s_1\theta+\theta_1) =|\theta|^{-N}       \tilde{w}(s_1\theta^{-1}+\theta_1), \qquad \tilde{u}(s_1\theta+\theta_1) =|\theta|^{-N} \tilde{u}(s_1\theta^{-1}+\theta_1). 
\end{equation}  
Observe that \[ -\infty < s_1\theta^{-1}+\theta_1< - \theta_1 \Longrightarrow - g(\theta_1) < s_1 \theta + \theta_1 < \theta_1. \] 
Note that $g(\theta_1) < \theta_0$ only if $\theta_0 < \frac{1}{\sqrt{3}}$. Consequently, by \eqref{conti-est-1}-\eqref{appli-lem-5.2-1} and the symmetry of $w$ and $u_{\sigma,1}$, we obtain 
\begin{equation}\label{whole-space}
\tilde{w}(\theta)>\tilde{u}(\theta) \quad \text{for all} \ \theta \in \mathbb{R},
\end{equation} 
which is again a contradiction to  $\|u\|_2 = \|u_{\sigma,1}\|_2=e^{\frac{N}{2} (\sigma-B_{N, \mathcal{L}})}.$ \\
\textbf{Case 2.} $\theta_0<\frac{1}{\sqrt{3}} < 1 < \theta_1.$ Define $\theta_{k+1} = \frac{\theta_{k}^2-1}{2 \theta_{k}}$ for $k \in \mathbb{N}.$ Then, we have $\theta_{k+1} < \frac{\theta_{k}}{2}$ and $\theta_k \to 0$ as $k \to \infty.$ Then, by repeating the same arguments as bove until $\theta_{k} <1$, we obtain \eqref{whole-space}, which is again a contradiction to  $\|u\|_2 = \|u_{\sigma,1}\|_2=e^{\frac{N}{2} (\sigma-B_{N, \mathcal{L}})}.$ 

Therefore, the assumption $u(0)>e^{\frac{N}{4} (\sigma-B_{N, \mathcal{L}})}B_{N,0}$ is impossible.  The case $u(0)<e^{\frac{N}{4} (\sigma-B_{N, \mathcal{L}})}B_{N,0}$ can be treated analogously, leading again to a contradiction. Hence, $u(0)=e^{\frac{N}{4} (\sigma-B_{N, \mathcal{L}})}B_{N,0}.$
\end{proof}
\begin{lemma}\label{lem:exact:norm}
    Let $u$ is a positive classical solution of \eqref{main:prob} with $\sigma \in \mathbb{R}$. Then $\|u\|_2 = e^{\frac{N}{4} (\sigma-B_{N, \mathcal{L}})}.$
\end{lemma}
\begin{proof}
By Theorem \ref{thm-radi-mono-prop}, $u$ is radially symmetric and monotonically decreasing about some point $z_0$ in $\mathbb{R}^N$ and 
\begin{equation}\label{beha-est}
      \lim_{|x| \to \infty} u(x) |x|^N = u_\infty. 
\end{equation} In view of Proposition \ref{pro:invariance} (i), we can assume that $z_0 =0.$
   Since, we know that $u_a(x) =a^{\frac{-N}{2}}u(x/a)$, $a>0$ is also a solution of the problem \eqref{main:prob} and satisfies $\|u_a\|_2 = \|u\|_2$. We can assume that
   \begin{equation}\label{ineq:est}
       u(0) > e^{\frac{N}{4} (\sigma-B_{N, \mathcal{L}})} B_{N,0} > u_\infty.
   \end{equation}
    Now, for an arbitrary unit vector $e \in \mathbb{R}^N$ and $t=1$, $x_0=0$, denote
    \[
    \tilde{w}(\theta) = u(\theta e), \quad \tilde{u}(\theta) = u_1(\theta e), \quad \theta \in \mathbb{R}.
    \]
By \eqref{solu:upperbound:estimate}, we know that $\|u\|_2 \leq e^{\frac{N}{4} (\sigma-B_{N, \mathcal{L}})}.$ Next, we show that $\|u\|_2=e^{\frac{N}{4} (\sigma-B_{N, \mathcal{L}})}$ and then by Theorem \ref{thm-uniqueness}, we obtain the required uniqueness claim. Suppose that $\|u\|_2 \in (0,e^{\frac{N}{4} (\sigma-B_{N, \mathcal{L}})}).$ Then, due to the scaling properties in Proposition \ref{pro:invariance}, radial symmetricity and continuity of $u$, and \eqref{ineq:est}, there exists $\theta_0 >0$ such that
\begin{equation}\label{eq:first-contact}
    \tilde{w}(\theta_0)=\tilde{u}(\theta_0),
    \qquad
    \tilde{w}(\theta) > \tilde{u}(\theta)
    \quad\text{for }|\theta|<\theta_0.
\end{equation}
and $0<r_1<r_2\leq+\infty$ for which
\begin{equation}\label{eq:negative-region}
    \tilde{w}(\theta)<\tilde{u}(\theta)
    \qquad\text{for }\theta \in (r_1,r_2), \qquad \tilde{w}(r_i)=\tilde{u}(r_i),\quad i=1,2.
\end{equation}
Next, we prove that \eqref{eq:negative-region} is impossible. Set $s:=\left(\frac{e^{\frac{N}{4} (\sigma-B_{N, \mathcal{L}})} B_{N,0}}{u_0(\theta_0)}\right)^{1/N}
      =\sqrt{1+\theta_0^2}.$
Applying Lemma~\ref{lem:self-dual-iden} with center
$x_0=\theta_0e$, we obtain
\begin{equation}\label{eq:inversion-at-iota0}
\tilde{w}(s\theta+\theta_0) =|\theta|^{-N} \tilde{w}(s\theta^{-1}+\theta_0), \qquad \tilde{u}(s\theta+\theta_0) =|\theta|^{-N} \tilde{u}(s\theta^{-1}+\theta_0).
\end{equation}
Observe that the condition $-\theta_0<s\theta^{-1}+\theta_0<\theta_0$ is equivalent to $-\infty<s\theta+\theta_0
    <-\theta_1$ where $\theta_1=g(\theta_0)$ and $g(\theta):= \frac{1-\theta^2}{2\theta}.$
Consequently, \eqref{eq:first-contact} and
\eqref{eq:inversion-at-iota0} imply
\begin{equation}\label{eq:transferred-inequality}
    \tilde{w}(\theta)>\tilde{u}(\theta)
    \quad\text{for }\theta \in (-\infty,-\theta_1) \cup (-\theta_0, \theta_0) \cup (\theta_1, \infty),
    \qquad
    \tilde{w}(\pm\theta_i)=\tilde{u}(\pm\theta_i), \, i=0,1.
\end{equation}
We distinguish four cases.\\
\textbf{Case 1: $\theta_0 \geq 1$.}
In this case, $0\geq\theta_1>-\theta_0.$ Hence \eqref{eq:first-contact} gives $\tilde{w}(\theta_1) > \tilde{u}(\theta_1),$
whereas \eqref{eq:transferred-inequality} gives $\tilde{w}(\theta_1)=\tilde{u}(\theta_1),$ which is a contradiction.\\
\noindent
\textbf{Case 2: $\frac{1}{\sqrt{3}}<\theta_0<1$.}
From \eqref{eq:transferred-inequality}, we have $0<\theta_1<\theta_0.$ Thus \eqref{eq:first-contact} yields $\tilde{w}(\theta_1)>\tilde{u}(\theta_1),$ again contradicting $\tilde{w}(\theta_1)=\tilde{u}(\theta_1)$ from \eqref{eq:transferred-inequality}.\\
\textbf{Case 3: $\theta_0=\frac{1}{\sqrt{3}}$.}
Then $\theta_1=\theta_0$. Combining
\eqref{eq:first-contact}, \eqref{eq:transferred-inequality}, and the radial symmetry of $u$ and $u_1$, we obtain
\[
    u(x)>u_1(x)
    \quad\text{for every }
    x\in\mathbb{R}^n\setminus
    \{x\in\mathbb{R}^n:|x|=\theta_0\}, \quad u(x)=u_1(x)
    \qquad\text{when }|x|=\theta_0.
\]
This contradicts the existence of the interval in \eqref{eq:negative-region} on which $\tilde{w}<\tilde{u}$. Now, the only case remaining is $\theta_0 <\frac{1}{\sqrt{3}}$. \\
\textbf{Case 4.} $\theta_0<\frac{1}{\sqrt{3}}$ and $\theta_1 \leq 1.$ Now, by repeating the above arguments with $s= \left(\frac{e^{\frac{N}{4} (\sigma-B_{N, \mathcal{L}})} B_{N,0}}{\tilde{u}(\theta_1)}\right)^\frac{1}{N}$ and applying Lemma~\ref{lem:self-dual-iden} with $x_0=\theta_1 e$, gives $s = s_1:=\sqrt{1+\theta_1^2}$ and \begin{equation}\label{appli-lem-5.2-1-modi} 
\tilde{w}(s_1\theta+\theta_1) =|\theta|^{-N}       \tilde{w}(s_1\theta^{-1}+\theta_1), \qquad \tilde{u}(s_1\theta+\theta_1) =|\theta|^{-N} \tilde{u}(s_1\theta^{-1}+\theta_1). 
\end{equation}  
Observe that \[ -\infty < s_1\theta^{-1}+\theta_1< - \theta_1 \Longrightarrow - g(\theta_1) < s_1 \theta + \theta_1 < \theta_1 \] 
Note that $g(\theta_1) < \theta_0$ only if $\theta_0 < \frac{1}{\sqrt{3}}$. Consequently, by \eqref{eq:transferred-inequality} and \eqref{appli-lem-5.2-1-modi} and the symmetry of $w$ and $u_1$, we obtain 
\begin{equation}\label{whole-space-new}
\tilde{w}(\theta)>\tilde{u}(\theta) \quad \text{for all} \ \theta \in \mathbb{R},
\end{equation} 
which is again a contradiction to  \eqref{eq:negative-region}.\\
\textbf{Case 5.} $\theta_0<\frac{1}{\sqrt{3}} < 1 < \theta_1.$ Define $\theta_{k+1} = \frac{\theta_{k}^2-1}{2 \theta_{k}}$ for $k \in \mathbb{N}.$ Then, we have $\theta_{k+1} < \frac{\theta_{k}}{2}$ and $\theta_k \to 0$ as $k \to \infty.$ Then, by repeating the same arguments as bove until $\theta_{k} <1$, we obtain \eqref{whole-space-new}, which is again a contradiction to  \eqref{eq:negative-region}. Consequently, the alternative $\|u\|_{2} < e^{\frac{N}{4} (\sigma-B_{N, \mathcal{L}})} $
cannot occur. Hence, the claim.
\end{proof}
\noindent 
\textbf{Proof of Theorem \ref{thm-uniqueness}:}
In view of Lemma \ref{end-point-fixing} and without loss of generality thanks to Proposition \ref{pro:invariance} $(i)$-$(ii)$ and Lemma \ref{lem:exact:norm}, assume that $u$ is a solution of \eqref{main:prob} satisfying $\|u\|_{2}=e^{\frac{N}{4} (\sigma-B_{N, \mathcal{L}})}$ and $u(0)=u_\infty=e^{\frac{N}{4} (\sigma-B_{N, \mathcal{L}})}B_{N,0}.$ It remains to prove that $ u \equiv u_{\sigma,1}$ in $\mathbb{R}^N.$ For an arbitrary unit vector $e\in\mathbb{R}^N$ and $t=1$, define \[ \tilde{w}(\theta)=u(\theta e), \qquad \tilde{u}(\theta)=u_{\sigma,1}(\theta e), \qquad \theta\in\mathbb{R}. \] 
Suppose, by contradiction, that there exists $\theta_0>0$ such that $\tilde{w}(\theta_0)>\tilde{u}(\theta_0).$ Since $u(0)=w_\infty= e^{\frac{N}{4} (\sigma-B_{N, \mathcal{L}})}B_{N,0},$ and applying Lemma \ref{lem:self-dual-iden} with $s=1$, $x_0=0$, and $x=\theta e$, yields \[
\tilde{w}(\theta)=|\theta|^{-N}\tilde{w}(\theta^{-1}), \qquad \tilde{u}(\theta)=|\theta|^{-N}\tilde{u}(\theta^{-1}), \qquad \theta \neq 0. 
\] 
Define 
\[ 
\widetilde{w}_0(\theta) = |\theta|^{-N} \tilde{w}(\theta^{-1}+\theta_0), \qquad \widetilde{u}_0(\theta) = |\theta|^{-N} \tilde{u}(\theta^{-1}+\theta_0). 
\]  
A direct computation shows that 
\begin{equation}\label{center-est}
    \widetilde{u}_0\!\left( -\frac{\theta_0}{1+\theta_0^2} \right) = e^{\frac{N}{4} (\sigma-B_{N, \mathcal{L}})}B_{N,0} (1+\theta_0^2)^\frac{N}{2} = (1+\theta_0^2)^N \widetilde{u}_0(\theta_0)\quad \text{and} \quad  \widetilde{w}_0\!\left( -\frac{\theta_0}{1+\theta_0^2} \right) = (1+\theta_0^2)^N \widetilde{w}_0(\theta_0). 
\end{equation}
Furthermore, 
\[ 
\widetilde{u}_0(0) = \lim_{\theta\to0} |\theta|^{-N} \tilde{u}(\theta^{-1}+\theta_0) =  e^{\frac{N}{4} (\sigma-B_{N, \mathcal{L}})}B_{N,0}, \quad  \widetilde{u}_0\!\left( -\frac{2\theta_0}{1+\theta_0^2} \right) = e^{\frac{N}{4} (\sigma-B_{N, \mathcal{L}})}B_{N,0}. 
\]  
Observe that $-\frac{\theta_0}{1+\theta_0^2}$ is the center of the radial function $\widetilde{u}_0$. In view of Theorem \ref{thm-radi-mono-prop}, let $\delta_0$ denote the center of $\widetilde{w}_0$. By \eqref{center-est} and since $\tilde{w}(\theta_0)>\tilde{u}(\theta_0)$, we obtain
\begin{equation}\label{center-lower-bound}
\widetilde{w}_0(\delta_0) \ge \widetilde{w}_0\!\left( -\frac{\theta_0}{1+\theta_0^2} \right) > \widetilde{u}_0\!\left( -\frac{\theta_0}{1+\theta_0^2} \right).\end{equation} 
Moreover, \[ (\widetilde{w}_0)_\infty := \lim_{\theta\to\infty} |\theta|^N\widetilde{w}_0(\theta) = \lim_{\theta\to\infty}  \tilde{w}(\theta^{-1}+\theta_0) = \tilde{w}(\theta_0), \quad (\widetilde{u}_0)_\infty = \lim_{\theta\to\infty} |\theta|^N \widetilde{u}_0(\theta) = \tilde{u}(\theta_0) \] which further implies $(\widetilde{w}_0)_\infty > (\widetilde{u}_0)_\infty.$ This together with \eqref{center-lower-bound}, we conclude that \[ \widetilde{w}_0(\delta_0) (\widetilde{w}_0)_\infty > \widetilde{u}_0\!\left( -\frac{\theta_0}{1+\theta_0^2} \right) (\widetilde{u}_0)_\infty = e^{\frac{N}{2} (\sigma-B_{N, \mathcal{L}})}B_{N,0}^2, \] 
which contradicts the claim in Lemma \ref{end-point-fixing}.  Hence, $\tilde{w}(\theta)\le \tilde{u}(\theta)$ for all $\theta\in\mathbb{R}$, or equivalently, $u(x)\le u_1(x)$ for all $x\in\mathbb{R}^N.$ If there exists $\theta_0>0$ such that \[ \tilde{w}(\theta_0)<\tilde{u}(\theta_0) \] is treated in exactly the same way and also leads to a contradiction. Since both $w$ and $u_{\sigma,1}$ satisfy $\|u\|_2 = \|u_{\sigma,1}\|_2= e^{\frac{N}{4} (\sigma-B_{N, \mathcal{L}})}$, it follows that $u \equiv u_{\sigma,1}$ in  $\mathbb{R}^N.$ Hence, the claim.
\qed
\appendix

\renewcommand{\thesection}{A}
\setcounter{section}{0}

\numberwithin{equation}{section}
\renewcommand{\theequation}{A.\arabic{equation}}
\setcounter{equation}{0}
\renewcommand{\thetheorem}{A.\arabic{theorem}}
\setcounter{theorem}{0}
\section{Appendix}\label{appendix}
Denote \begin{equation}\label{constants-value}
        B_{N,0} = \left(\frac{\Gamma(N)}{\pi^{\frac{N}{2}}\Gamma\left(\frac{N}{2}\right)}\right)^\frac{1}{2} \quad \text{and} \quad B_{N,1} = B_{N,0} \left[ \ln 2 + 2 \psi\left(\frac{N}{2}\right) - \psi(N)\right].
    \end{equation}
\begin{lemma}\label{lem:norm-compu}
    For any $t>0$, $x_0 \in \mathbb{R}^N$ and $s \in [0, 1)$, the following holds:
\begin{equation}\label{est:crit-choq-2}
    \|U_{s,t}\|_{2_s^\ast}^{2_s^\ast} = \frac{\pi^{\frac{N}{2}} \Gamma\left(\frac{N}{2}\right)}{\Gamma(N)} \quad \text{and} \quad \|U_{s,t}\|_{2}^{2} = \frac{\pi^{\frac{N}{2}} t^{2s} \Gamma\left(\frac{N-4s}{2}\right)}{\Gamma(N-2s)}.
\end{equation}
In particular, $\|U_{0,1}\|_2 = \|U_{0,t}\|_2 =\frac{1}{B_{N,0}}.$
\end{lemma}
\begin{proof}
By applying change of variables multiple times, we obtain
    \[
    \begin{split}
        \|U_{s,t}\|_{2_s^\ast}^{2_s^\ast} &= \int_{\mathbb{R}^N} \left(\frac{t}{t^2+|x-x_0|^2}\right)^{N} ~dx = \int_{\mathbb{R}^N} t^{-N} \left(\frac{t}{1+\left|\frac{x-x_0}{t}\right|^2}\right)^{N} ~dx \\
        & = \int_{\mathbb{R}^N} \left(\frac{1}{1+|y|^2}\right)^{N} ~dy =  \|U_{s,1}\|_{2_s^\ast}^{2_s^\ast} = \omega_{N-1} \int_0^\infty \frac{y^{N-1}}{(1+y^2)^{N}} ~dy\\
        & = \frac{\omega_{N-1}}{2} \int_0^\infty \frac{y^{\frac{N}{2}-1}}{(1+y)^{N}} ~dy = \frac{\omega_{N-1}}{2} B\left(\frac{N}{2}, \frac{N}{2}\right) = \frac{\pi^{\frac{N}{2}} \Gamma \left(\frac{N}{2}\right)}{\Gamma(N)}
    \end{split}
    \]
and
    \[
    \begin{split}
        \|U_{s,t}\|_2^2 &= \int_{\mathbb{R}^N} \left(\frac{t}{t^2+|x-x_0|^2}\right)^{N-2s} ~dx = \int_{\mathbb{R}^N} t^{-2N+4s} \left(\frac{t}{1+\left|\frac{x-x_0}{t}\right|^2}\right)^{N-2s} ~dx \\
        & = t^{2s} \int_{\mathbb{R}^N} \left(\frac{1}{1+|y|^2}\right)^{N-2s} ~dy = t^{2s} \|U_{s,1}\|_2^2 = t^{2s}\omega_{N-1} \int_0^\infty \frac{y^{N-1}}{(1+y^2)^{N-2s}} ~dy\\
        & = \frac{t^{2s} \omega_{N-1}}{2} \int_0^\infty \frac{y^{\frac{N}{2}-1}}{(1+y)^{N-2s}} ~dy = \frac{t^{2s} \omega_{N-1}}{2} B\left(\frac{N}{2}, \frac{N}{2}-2s\right) = \frac{t^{2s} \pi^{\frac{N}{2}} \Gamma \left(\frac{N}{2}-2s\right)}{\Gamma(N-2s)}
    \end{split}
    \]
where $B(\cdot, \cdot)$ denotes the beta function and $\omega_{N-1} = \frac{2 \pi^\frac{N}{2}}{\Gamma \left(\frac{N}{2}\right)}$. 
\end{proof}
\begin{lemma}\label{lem:constants}
    Let $B_{N,s}$ be defined in \eqref{frac-choquard-constant}. Then $\lim_{s \to 0^+} B_{N,s} = B_{N,0}$ and 
\[B_{N,s} = B_{N,0} + s B_{N,1} + o(s) \ \text{as} \ s \to 0^+.\] 
\end{lemma}
\begin{proof} 
Denote
\[
D_{N,s} := \frac{2^{2s}\Gamma\left(\frac{N+2s}{2}\right)\Gamma(N-2s)}
{\pi^{N/2}\Gamma\left(\frac{N-2s}{2}\right)\Gamma\left(\frac{N-4s}{2}\right)} \quad \text{such that} \ B_{N,s}=D_{N,s}^{1/2}.
\]
Differentiating with respect to \(s\), we obtain
\[
\frac{d}{ds}B_{N,s} =
\frac{1}{2} D_{N,s}^{-1/2}\frac{dD_{N,s}}{ds} = \frac{1}{2} B_{N,s}\frac{D_{N,s}'}{D_{N,s}}
\]
with the notation $D_{N,s}'=\frac{dD_{N,s}}{ds}$. Now, by taking logarithms yields
\[
\begin{split}
\log D_{N,s} & =
2s\log 2
+
\log\Gamma\left(\frac{N+2s}{2}\right)
+
\log\Gamma(N-2s) -\frac N2\log\pi-
\log\Gamma\left(\frac{N-2s}{2}\right) - \log\Gamma\left(\frac{N-4s}{2}\right).
\end{split}
\]
Using the identity $\psi(x) = \frac{d}{dx}\log\Gamma(x)$ and differentiating with respect to $s$, we obtain
\[
\frac{D_{N,s}'}{D_{N,s}}
=
2\log2
+
\psi\left(\frac{N+2s}{2}\right)
-2\psi(N-2s)
+
\psi\left(\frac{N-2s}{2}\right)
+
2\psi\left(\frac{N-4s}{2}\right).
\]
Therefore,
\[
\frac{d}{ds}B_{N,s}
=
B_{N,s}
\left[
\log2
+\frac{1}{2}\psi\left(\frac{N+2s}{2}\right)
-\psi(N-2s)
+\frac{1}{2}\psi\left(\frac{N-2s}{2}\right)
+\psi\left(\frac{N-4s}{2}\right)
\right].
\]
Finally, by inserting $s=0$, we obtain the required claim.
\end{proof}
Next, we recall the Hardy-Littlewood Sobolev inequality \cite[Theorem 3.1]{Lieb-1983} and  Pitt's inequality by Beckner \cite[Theorem 3]{Beckner-1995} and Beckner's entropy inequality \cite[Theorem~2]{Beckner-1993}, adapted to the setting of our problem:
\begin{theorem}
    Let $N> 4s$ and $u \in L^{2_s^\ast}(\mathbb{R}^N).$ Then there exists a sharp constant $C_{N,s}$, independent of $u$, such that
    \begin{equation}\label{HLS-inequality}
        \int_{\mathbb{R}^N} \int_{\mathbb{R}^N} \frac{u^2(y)}{|x-y|^{4s}}  u^2(x) ~dy ~dx \leq C_{N,s} \|u\|_{2_s^\ast}^4 \quad \text{with} \quad C_{N,s} = \pi^{2s} \frac{\Gamma\left(\frac{N-4s}{2}\right)}{\Gamma\left(N-2s\right)} \left(\frac{\Gamma\left(\frac{N}{2}\right)}{\Gamma(N)}\right)^{-1+\frac{4s}{N}}.
    \end{equation}
Moreover, the equality in \eqref{HLS-inequality} is achieved uniquely (up to a multiplication by a constant) by the function $U_{s,t}.$ 
\end{theorem} 
\begin{lemma}[Pitt's inequality]\label{lem:Pitt's ineq}
    Let $u \in \mathbb{H}_{\ln}(\mathbb{R}^N)$ with $\|u\|_2 =1.$
    Then, we have
\begin{equation}\label{Pitts:ineq}
    \frac{4}{N} \int_{\mathbb{R}^N} |u(x)|^2 \ln |u(x)| ~d\xi \leq  \int_{\mathbb{R}^N} \ln |\xi|^2 |\mathcal{F}_0(u)|^2 ~dx - \frac{4 G_N}{N}
\end{equation}
where 
$G_N = \frac{N}{2} \psi\left(\frac{N}{2}\right) - \frac{N}{4} \ln \pi - \frac{1}{2} \ln\left(\frac{\Gamma(N)}{\Gamma\left(\frac{N}{2}\right)}\right)$. Moreover, the extremal functions are given up to conformal automorphism by $u(x)= A \left(1+ |x|^2\right)^{\frac{-N}{2}}$ for some $A  \in \mathbb{R}^+.$
\end{lemma}
\begin{lemma}[Beckner's entropy inequality]\label{Beckner-ineq}
    Let $u \in \mathbb{H}_{\ln}(\mathbb{R}^N)$ such that $\|u\|_2 =1.$ Then, 
    \begin{equation}\label{LOg-HLS-ineq-2}
        \int_{\mathbb{R}^N}  \left(\ln \left(\frac{1}{|x|^4}\right) \ast u^2\right) u^2(x) ~dx \leq \frac{2C_N}{N} + \frac{8}{N} \int_{\mathbb{R}^N} u^2(x) \ln |u|~dx
    \end{equation}
    where $C_N$ is defined as $C_N = N \ln \pi + N\left(\psi(N) - \psi\!\left(\frac{N}{2}\right)\right) - 2\ln\!\left(\frac{\Gamma(N)}{\Gamma\!\left(\frac{N}{2}\right)}\right).$ Moreover, the extremal functions are given up to conformal automorphism by $u(x)= A \left(1+ |x|^2\right)^{\frac{-N}{2}}$ for some $A  \in \mathbb{R}^+.$
\end{lemma}
\section*{Acknowledgment} The first author gratefully acknowledges the financial support of the Anusandhan National Research Foundation (ANRF), India, under Grant No. ANRF/ARGM/2025/000272/MTR. This work was initiated during the second author's research visit to the Department of Mathematical Sciences, IIT (BHU), Varanasi, in February 2026, and was progressed during the first author's research visit to LMAP, UPPA, in July 2026. The authors sincerely acknowledge the warm hospitality and support extended by both institutions during these visits.

\end{document}